%% file: crb_mult_reelle.tex
\documentclass[a4paper,10pt,twoside,notitlepage]{article}
\usepackage[utf8]{inputenc}
\usepackage[T1]{fontenc}
\usepackage[frenchb]{babel}
\usepackage[a4paper,top=3cm,left=3cm, right=3cm,bottom=3cm, twoside]{geometry}
\author{Ivan {\sc Boyer}\footnote{Sous la direction de Jean-François
    {\sc Mestre}}\\Université Paris Diderot}
\date{\today}
\title{Courbes à multiplication réelle\\ par des sous-corps de cyclotomiques}

\input{commandes}

\begin{document}

\maketitle

\begin{abstract}
In \emph{Endomorphism Algebras of Jacobians},~\cite{ellen},  Ellenberg gives group theory tools to construct jacobians of curves with real multiplication. He shows the existence of curves and family of curves with real multiplication by subfields of cyclotomic fields. Among them, some are already known such as the family of Mestre in~\cite{mestrehyp}, or the family of Tautz, Top and  Verberkmoes in~\cite{top}. In this article, we  give explicit families of real multiplication curves, for each case Ellenberg showed their existence.
\end{abstract}

\tableofcontents

\include{intro}

\include{corps}

\bibliographystyle{bib}
\phantomsection
\addtotoc{Bibliographie}
\bibliography{these}

\end{document}

%% file: commandes.tex
\usepackage{lmodern}
\rmfamily
\DeclareFontShape{T1}{lmr}{b}{sc}{<->ssub*cmr/bx/sc}{}
\DeclareFontShape{T1}{lmr}{bx}{sc}{<->ssub*cmr/bx/sc}{}
\usepackage{slantsc}
\usepackage{calrsfs}
\usepackage{stmaryrd}
\usepackage{amsmath,amssymb,amsthm,textcomp}
\usepackage[mediummath]{nccmath}
\usepackage{latexsym,multirow,xspace}
\usepackage{mathtools,etoolbox}
\newcommand{\trp}[1]{\kern-.1em\prescript{t\kern-.08em}{}{#1}}
\usepackage{tikz}
\usetikzlibrary{calc, intersections,through,arrows,positioning}
\usepackage{xspace}

\newcommand{\Mc}{\mathcal M}

\newcommand{\Oc}{\mathcal O}

\newcommand{\Zz}[1]{{\mathbb Z/#1\mathbb Z}}
\newcommand{\Z}{\mathbb Z}

\newcommand{\Q}{\mathbb Q}

\newcommand{\C}{\mathbb C}
\newcommand{\F}{\mathbb F}
\newcommand{\Pp}{\mathbb P}

\newcommand{\Ri}{Riemann\xspace}
\newcommand{\Hu}{Hurwitz\xspace}

\newcommand{\ti}{\-{--}}
\newcommand{\cad}{c'est-\`a-dire\xspace}

\newcommand{\cf}{\textit{cf.}\xspace}

\usepackage{amsthm}

\newtheorem{defi}{D\'efinition}[section]
\newtheorem{theo}[defi]{Th\'eor\`eme}
\newtheorem*{theos}{Th\'eor\`eme}
\newtheorem{prop}[defi]{Proposition}

\newtheorem{nota}[defi]{Notation}

\theoremstyle{definition}

\newtheorem{expl}[defi]{Exemple}
\newtheorem{rmq}[defi]{Remarque}
\newtheorem{algo}[defi]{Algorithme}
\usepackage{algpseudocode}

\DeclareMathOperator{\Div}{Div}
\let\div\relax
\DeclareMathOperator{\div}{div}

\DeclareMathOperator{\Res}{Res}
\DeclareMathOperator{\End}{End}
\DeclareMathOperator{\Jac}{Jac}
\DeclareMathOperator{\Aut}{Aut}

\DeclareMathOperator{\D}{d\!}
\DeclareMathOperator{\ord}{ord}

\DeclareMathOperator{\carac}{car}

\usepackage{fixmath}

\renewcommand\epsilon\varepsilon
\renewcommand\emptyset\varnothing
\renewcommand{\phi}{\varphi}
\usepackage[breaklinks=true,%
hyperfootnotes=false,%
colorlinks=true,%
linkcolor=blue,%
citecolor=green,%
urlcolor=red,%
pagebackref=true,%
hyperindex%
]{hyperref}
\backreffrench

\usepackage{xstring}
\renewcommand*{\backref}[1]{%
\StrExpand[0]{#1}{\arg}%
\let\bs\backrefxxx%
\renewcommand*\backrefxxx[3]{}%
\IfSubStr{\arg}{,}{\def\fin{p}}{\def\fin{}}%
\let\backrefxxx\bs%
({\itshape cf.} p\fin.~\arg)
}%
\newcommand{\urlb}[1]{\href{http://#1}{\url{#1}}}
\usepackage[tight]{shorttoc}
\usepackage[titles]{tocloft}

\usepackage[symbol]{footmisc}
\DefineFNsymbols{symb}[text]{{($\star$)} {($\diamond$)}  {($\ddagger$)} {($\!\circledast\!$)} {($\mathsection$)} {($\star\star$)} {($\diamond\diamond$)}  {($\ddagger\ddagger$)} {($\!\circledast\kern-.6pt\circledast\!$)} {($\mathsection\mathsection$)} }
\setfnsymbol{symb}
\renewcommand\labelitemi\textbullet
\renewcommand{\FrenchLabelItem}\textbullet

\usepackage{alphalph}
\makeatletter
\@addtoreset{footnote}{section}
\makeatother

\newif\ifsommaire
\sommairefalse

\newbool{nosommaire}
\boolfalse{nosommaire}

\newcommand{\addtotoc}[1]{%

{\ifbool{nosommaire}%
{\addtocontents{toc}{\protect\ifsommaire\protect\else}
\addcontentsline{toc}{section}{#1}%
\addtocontents{toc}{\protect\fi}%
}
{\addcontentsline{toc}{section}{#1}}%
}
}

\usepackage{relsize}

\newcommand{\WS}{Weierstrass\xspace}

\newcommand{\va}{vari\'et\'e ab\'elienne\xspace}

\newcommand{\tcm}{type-\textsc{cm}\xspace}

\newcommand{\RR}{Riemann-Roch\xspace}

\renewcommand\theta{{\mathbold\vartheta}}

\usepackage{multirow}
\usepackage{tabularx}
\newenvironment{itemize0}[1]%
   {%
   \begin{list}{--}%
   {\setlength{\leftmargin}{0cm}%
   }%
   #1%
   }%
   {%
   \end{list}%
   }%

\usepackage{float}
\usepackage{caption}
\usepackage[titles]{tocloft}

\usepackage{algpseudocode}

\floatstyle{ruled}
\newcommand{\listeprog}{Liste des algorithmes}
\newlistof[section]{algo}{alg}{\listeprog}
\newfloat{algo}{h!}{alg}%
\floatname{algo}{Algorithme}
\restylefloat{table}
\restylefloat{figure}
\cftsetindents{algo}{1.5em}{2.3em}

\addto\captionsfrenchb{}

\usepackage{microtype}

%% file: intro.tex
\section*{Introduction}
\addtotoc{Introduction}
Dans cet article, on s'intéresse aux familles de courbes à multiplication réelle dont J. Ellenberg montre l'existence dans \emph{Endomorphism Algebras of Jacobians}~\cite{ellen}. Pour chaque famille, nous montrons comment construire explicitement de telles courbes sur des extensions convenables de $\Q$ ou lorsque le calcul est trop lourd, sur des corps finis. Plus précisément, nous obtenons 
\begin{itemize}\item une famille à trois paramètres de courbes à multiplication réelle par $\Q(\zeta_l^+)$,\item une famille à un paramètre de courbes à multiplication réelle par $\Q(\zeta_l^{(4)})$\item une famille à un paramètre de courbes à multiplication réelle par $\Q(\zeta_l^{(6)})$,\item  deux courbes à multiplication réelle respectivement par $\Q(\zeta_l^{(8)})$ et $\Q(\zeta_l^{(10)})$.
\end{itemize}

On peut aussi se reporter au tableau final de la section~\ref{chap2:sec:fin}, un peu plus détaillé. Parmi ces familles, on s'attache à donner des exemples définis sur $\Q$.
\begin{theos}
Il existe une famille explicite et définie sur $\Q$, à deux paramètres, de courbes à multiplication réelle par $\Q(\zeta_l^+)$.
\end{theos}

Par exemple, on obtient à nouveau (\cf p.~\pageref{q5plus}) une famille à deux paramètres de courbes hyperelliptiques de genre 2 à multiplication réelle par $\Q(\zeta_5^+)$. Dans le cas $l=7$, nous obtenons la famille à deux paramètres suivante.
\begin{theos}
L'équation quartique suivante, en $U$ et $V$ et à deux paramètres $r$ et $t$, est à multiplication réelle par $\Q(\zeta_7^+)$. 

{\relpenalty=0000
\binoppenalty=0000
\footnotesize
\noindent$- ( 2tr-r-{t}^{3} )  ( -{r}^{3}-40{t}^{2}{r}^{3}+
10t{r}^{3}-80{r}^{3}{t}^{4}+80{t}^{3}{r}^{3}+32{r}^{3}{t}^{5}+
16{r}^{2}{t}^{7}-56{r}^{2}{t}^{6}+64{r}^{2}{t}^{5}-30{r}^{2}{t
}^{4}+5{r}^{2}{t}^{3}+12r{t}^{8}-16r{t}^{7}+5r{t}^{6}+{t}^{9}
 ) +2r ( 2t-1 ) ^{2} ( {r}^{2}-4t{r}^{2}+4
{t}^{3}r+4{r}^{2}{t}^{2}-10r{t}^{4}+2{t}^{6}+4r{t}^{5}
 ) U+{r}^{2} ( 2t-1 ) ^{4}{U}^{2}-r ( 2t-1
 ) ^{2}{U}^{3}+r ( 2t-1 ) ^{2} ( -2{r}^{3}-6
{r}^{2}{t}^{2}+12t{r}^{3}+24{r}^{2}{t}^{3}-24{t}^{2}{r}^{3}-12
r{t}^{5}-24{r}^{2}{t}^{4}+16{t}^{3}{r}^{3}+28r{t}^{6}-5{t}^{
8}-8r{t}^{7} ) V-2{t}^{2} ( {t}^{3}+2r-6tr+4{t}^{
2}r ) ^{2}UV-t ( {t}^{3}+2r-6tr+4{t}^{2}r ) 
 ( {r}^{2}-4t{r}^{2}+4{t}^{3}r+4{r}^{2}{t}^{2}-10r{t}^{4}
+2{t}^{6}+4r{t}^{5} ) {V}^{2}-2rt ( 2t-1 ) ^{
2} ( {t}^{3}+2r-6tr+4{t}^{2}r ) U{V}^{2}+t ( {t}
^{3}+2r-6tr+4{t}^{2}r ) {U}^{2}{V}^{2}+2r ( 2t-1
 ) ^{2} ( {r}^{2}-4t{r}^{2}+4{t}^{3}r+4{r}^{2}{t}^{2}
-10r{t}^{4}+2{t}^{6}+4r{t}^{5} ) {V}^{3}+ ( {r}^{2}-4
t{r}^{2}+4{t}^{3}r+4{r}^{2}{t}^{2}-10r{t}^{4}+2{t}^{6}+4r{
t}^{5} ) U{V}^{3}+{t}^{2} ( {t}^{3}+2r-6tr+4{t}^{2}r
 ) ^{2}{V}^{4}
.$
}%
\end{theos}
Parmi cette famille, on exhibe l'exemple :
\[v^3(2u-1)-u(u^3+2u^2-u-1)=0,\] 
qui est à multiplication complexe par $\Q(\zeta_7^+,i\sqrt3)$. En utilisant une autre construction, on donne aussi deux autres familles à un et deux paramètres de telles courbes.
\begin{theos}
Les courbes données par les équations quartiques à un paramètre $s$, 
\[2v + u^3 +(u+1)^2+s\bigl((u^2+v)^2 - v (u+v)(2u^2-uv + 2v)\bigr),\]
et deux paramètres $s$ et $t$,

{
\relpenalty=0000
\binoppenalty=0000
\noindent $- ( s+t ) ^{2}+2 ( s+t ) sv+ ( -{s}^{2}+3
{t}^{2}-t ) v^{2}+ ( 6{t}^{2}-2t-2{s}^{2}
 ) v^{3}+2 ( s+t ) ^{2}u+ ( 6{t}^{2}-2t-
2{s}^{2} ) uv^{2}+ ( -{s}^{2}+3{t}^{2}-t ) uv
^{3}- ( s+t )  ( s-t ) u^{2}+ ( 2t+2{
s}^{2}-6{t}^{2} ) u^{2}v+ ( t-3{t}^{2}+{s}^{2}
 ) u^{2}v^{2}+ ( s+t )  ( s-t ) u^{3
}+ ( 2t+2{s}^{2}-6{t}^{2} ) u^{3}v+ ( -{s}^{2}
+3{t}^{2}-t ) u^{4}.$
}

\noindent
sont à multiplication réelle par $\Q(\zeta_7^+)$.
\end{theos}

\noindent
On donne enfin, pour les exemples définis sur $\Q$, la courbe quartique
\[u^3+2u^2v+2u^2-uv^3-2uv^2-2uv-2u+v^4+v^2-v+2\]
qui est à multiplication réelle par $\Q(\zeta_{13}^{(4)})$ ; c'est un exemple du résultat plus général suivant.
\begin{theos}
Soit $l\equiv1\mod 4$ un nombre premier. Il existe une courbe, définie sur $\Q(i)$, à multiplication réelle par $\Q(\zeta_{l}^{(4)})$.
\end{theos}

%% file: corps.tex
\section{Recouvrements, monodromie et multiplication réelle}\label{sec:ellen1}

\paragraph{Multiplication réelle} Avant de décrire ces familles, commençons par fixer les notions et les notations principales pour la suite. Voici ce que l'on entend par multiplication réelle. 

\begin{defi}[Multiplication réelle]
\label{mult:reelle}
On dit qu'une \va $A$ de dimension $g$ est à \emph{multiplication réelle} s'il existe un corps de nombres totalement réel $F$ de dimension $g$ tel que 
\[
F\hookrightarrow\End_\Q(A).
\]  
De plus, on dit plus simplement qu'une courbe de genre $g$ est à \emph{multiplication réelle} par $F$ si sa jacobienne l'est.
\end{defi}  

Sur les corps finis de caractéristique $p$, on peut vérifier la multiplication complexe grâce au polynôme caractéristique du Frobenius, $\pi(t)$. Ici, comme on s'intéresse à la multiplication réelle, il est avantageux de regarder le polynôme caractéristique de la trace du Frobenius, donné par 
$\Res_t(\pi(t),rt-t^2-p)$.

\paragraph{Recouvrements} Le point de départ de J. Ellenberg consiste à considérer une courbe algébrique projective, non-singulière, $Y$, définie sur un corps algébriquement clos, $k$, telle que son groupe d'automorphismes soit \og important \fg{}. En notant $G\subset\Aut(Y)$ un sous-groupe, on a une action naturelle
\[
\Q[G]\rightarrow\End_\Q(\Jac(Y)).
\]
On note $C$ la courbe quotient non singulière $Y/G$,  associée aux éléments $G$\ti invariants de $k(Y)$ ; on considère alors un sous-groupe $H$ de $G$ et on définit l'élément $\pi_H\in\Q[G]$, 
\[\pi_H:=\frac1{|H|}\sum_{h\in H}H,\]
et l'algèbre $\Q[H\backslash G/H]:=\{\pi_Hg\pi_H,\ g\in G\}$, sous-algèbre de $\Q[G]$ dite de Hecke. On note enfin $X=Y/H$, la courbe quotient non singulière, associée aux éléments $H$\ti invariants de $k(Y)$. 

L'action naturelle  de $\Q[G]$ se restreint en une action de $\Q[H\backslash G/H]$ sur $\Jac(X)$. On note, suivant J. Ellenberg, $\mathcal H_{X/C}$ l'image de cette algèbre dans $\End_\Q(\Jac(X))$. On a le diagramme suivant qui résume la situation.
\begin{equation}\label{diag:chap:2}
\begin{tikzpicture}[shorten >=1pt,node distance=1.8cm,auto,baseline={([yshift=2.16cm]current bounding box.south)}]
\node (Y) {$Y$};
\node (X) [below left=of Y] {$X$};
\node (C) [below right=of X] {$C$};
\draw[-latex'] (Y) to node {$k(C)=k(Y)^G$} (C);
\draw[-latex'] (Y) to node[swap] {$k(X)=k(Y)^H$} (X);
\draw[-latex'] (X) to (C);
\end{tikzpicture}
\end{equation}
\paragraph{Branchements et monodromie}
En considérant l'injection $\End_\Q(\Jac(X))\hookrightarrow\End_\Q(H_1(X,\Q_l))$ pour un entier premier $l\neq\carac(k)$, on peut calculer $\mathcal H_{X/C}$ grâce à la représentation de $\Q[H\backslash G/H]$ sur $H_1(X,\Q_l)$. Cela permet à J. Ellenberg de relier $\mathcal H_{X/C}$ au branchement de $Y\rightarrow C$. Grâce à une utilisation fine de la formule de \Ri-\Hu, J. Ellenberg exhibe des cas de monodromie de ce recouvrement, que nous détaillons ci-dessous, permettant d'obtenir des courbes $X$ à multiplication réelle.

Afin d'énoncer ce résultat, nous avons encore besoin de quelques définitions et notations.
\begin{defi}[Groupes métacycliques]
On dit qu'un groupe est \emph{métacyclique} s'il possède un sous-groupe normal cyclique dont le quotient est lui-aussi cyclique.
\end{defi}

Par la suite, on considère un type plus restreint de groupes métacycliques. Soit $l$ un entier premier impair, $n$ un entier divisant $l-1$ et $k\in(\Zz l)^*$ d'ordre $n$.

\begin{defi}[Groupes métacycliques $G_{l,n}$]
On définit le groupe \emph{métacyclique} $G_{l,n}$ par deux générateurs $\alpha,\sigma$ d'ordres respectifs $l$ et $n$ avec la relation de conjugaison $\alpha\sigma\alpha^{-1}=\sigma^k$ :
\[G_{l,n}:=\langle\alpha,\sigma,\ \sigma^l=\alpha^n=1,\alpha\sigma\alpha^{-1}=\sigma^k\rangle.\]
\end{defi}
On reconnaît bien évidemment dans cette définition les groupes diédraux, pour $n=2$.
Du fait que l'on a choisi $n\mid l-1$, on peut considérer des sous-corps d'indice $n$ des corps cyclotomiques.
\begin{nota}[Sous-corps des cyclotomiques]
Soit $l$ un entier premier impair. On note $\zeta_l$ une racine $l$-ième de l'unité et $\Q(\zeta_l)$ le corps cyclotomique de degré $l-1$. Pour $n$ divisant $l-1$, on note $\Q(\zeta_l^{(n)})$ le sous-corps d'indice $n$ de $\Q(\zeta_l)$. Il est engendré par 
\[\zeta_l^{(n)}:=\sum_{i=0}^{n-1}\zeta_l^{k^i}\]
où $k\in(\Zz{l})^*$ est d'ordre $n$. On rappelle que l'on note de manière abrégée dans la suite $\Q(\zeta_l^+):=\Q(\zeta_l^{(2)})$, le sous-corps totalement réel maximal de $\Q(\zeta_l)$.
\end{nota}
En particulier, pour tout $n$ impair divisant $l-1$, le corps $\Q(\zeta_l^{(n)})$ est totalement réel, de dimension $\tfrac{l-1}n$.
Enfin, afin de décrire les branchements et la monodromie, on conserve la notation de J. Ellenberg suivante.
\begin{nota}\label{nota:types}
Soit $Y\rightarrow\Pp^1$ un recouvrement de groupe de Galois $G_{l,n}$, possédant $r$ points de branchement. On note $g_1,\dots,g_r\in G_{l,n}$ la monodromie de ce recouvrement en chacun des points de branchement. On note alors $d_i=0$ si l'ordre de $g_i$ est $l$, $d_i=\tfrac n{\ord{g_i}}$ sinon.

\noindent
On dit enfin qu'un tel recouvrement est de type $(d_1,\dots,d_r)$.
\end{nota}
On est maintenant en mesure d'énoncer un des principaux résultats que J. Ellenberg expose dans son article,~\cite{ellen}, reliant les données de branchement de $Y\rightarrow C$ à $\mathcal H_{X/C}$, dans le cas où $G$ est un groupe métacyclique et $C=\Pp^1$.
\begin{theo}[Ellenberg]
Soient $l$ un entier premier impair, $n$ un entier pair, divisant $l-1$. Soit $Y\rightarrow \Pp^1$ un recouvrement de groupe de Galois métacyclique, $G_{l,n}$. On note $H$ le sous-groupe de $G_{l,n}$ engendré par $\alpha$, d'ordre $n$ et $X=Y/H$.

Alors, $\Jac(X)$ est à multiplication réelle par $\Q(\zeta_l^{(n)})$ si et seulement si le recouvrement $Y\rightarrow \Pp^1$ est parmi les types listés dans la table~\ref{table:types} ci-dessous.
\end{theo}
\vspace*{-10pt}
\begin{table}[h]
\centering
\vspace*{5pt}
\begin{tabular}{|c|c|}\hline$n$&\textit{types}\\\hline
2&\hyperref[type:0011]{$(0,0,1,1)$}, \hyperref[type:01111]{$(0,1,1,1,1)$},\hyperref[type:111111]{$(1,1,1,1,1,1)$}\\
4&\hyperref[type:011]{$(0,1,1)$},\hyperref[type:1122]{$(1,1,2,2)$}\\
6&\hyperref[type:112]{$(1,1,2)$},\hyperref[type:2233]{$(2,2,3,3)$}\\
8&\hyperref[type:114]{$(1,1,4)$}\\
10&\hyperref[type:125]{$(1,2,5)$}\\\hline
\end{tabular}
\vspace*{5pt}
\caption{Types de recouvrement de groupe $G_{l,n}$}
\label{table:types}
\end{table}

La suite de ce chapitre consiste à donner des exemples explicites de recouvrements correspondant à ces types. L'idée consiste à construire $Y$ en deux temps, en passant par une courbe $Z$, dont les recouvrements $Y\rightarrow Z\rightarrow \Pp^1$ sont de degrés respectifs $l$ et $n$ et de groupes de Galois cycliques $\Zz l$ et $\Zz n$.

Si le type considéré de la table~\ref{table:types} possède des zéros, on cherche des revêtements $Y\rightarrow Z$ ramifiés, et non-ramifiés dans le cas contraire.

On commence dans tous les cas par chercher un recouvrement $Z\rightarrow\Pp^1$ de groupe de Galois cyclique d'ordre $n$. On conserve la même notion de type de ramification, en adaptant la notation~\ref{nota:types} sans difficulté. On cherche donc des revêtements de degré $n$ dont les types figurent dans la table~\ref{table:types}, hormis les 0. On construit ensuite un deuxième recouvrement $Y\rightarrow Z$, de groupe de Galois $\Zz l$, de façon à obtenir le bon type et, par composition, le groupe de Galois $G_{l,n}$, ce qui est en fait le plus contraignant. Cette stratégie se résume sur le diagramme suivant.

\newsavebox{\Crv}
\sbox{\Crv}{$\Pp^1$}
\newlength{\dimCrv}
\settowidth{\dimCrv}{\usebox{\Crv}}

\begin{figure}[H]
\vspace*{10pt}
\begin{center}
\begin{tikzpicture}[shorten >=1pt,node distance=1.8cm,auto,baseline={([yshift=2.16cm]current bounding box.south)}]
\node (Y) {\makebox[\dimCrv]{$Y$}};
\node (X) [below left=of Y] {$X$};
\node (Z) [below right=of Y] {$Z$};

\node (P) [below right=of X] {$\Pp^1$};
\draw[-latex'] (Y) to node {$\Zz l$} (Z);
\draw[-latex'] (Y) to node[swap] {$\Zz n$} (X);
\draw[-latex'] (Z) to node {$\Zz n$} (P);
\draw[-latex'] (X) to node{} (P);
\draw[-latex'] (Y) to node{$G_{l,n}$} (P);
\end{tikzpicture}
\caption{Revêtements métacycliques galoisiens}
\end{center}
\end{figure}
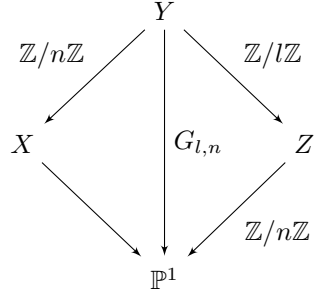

\section{Multiplication réelle par \texorpdfstring{$\Q(\zeta_l^+)$}{Q(zeta\_l+)}}
Sauf mention contraire, on se place dans cette section sur le corps de base $k=\Q$.
Le cas $n=2$ a déjà beaucoup été étudié, et l'on renvoie par exemple à W. Tautz, J. Top et A. Verberkmoes~\cite{top} et J.-F. Mestre~\cite{mestrehyp} pour des familles à un, respectivement deux, paramètres de courbes hyperelliptiques à multiplication réelle par $\Q(\zeta_l^+)$. On retrouve ici les familles de~\cite{top} et on donne d'autres familles différentes. Le type~\ref{type:111111} correspond par ailleurs aux résultats de A. Brumer.
\subsection{Type (0,0,1,1)}\label{type:0011}
Comme on l'a expliqué à la fin de le section précédente, on commence par chercher un recouvrement de la droite projective $Z\rightarrow\Pp^1$, branché en deux points, de \og type \fg{} (1,1), c'est à dire simplement un recouvrement de degré 2 ramifié en deux points. Par la formule d'Hurwitz, cela impose\vspace*{-7pt}
\[g(Z)=\tfrac12\big(2+2(2g(\Pp^1)-2)+(2-1)+(2-1)\big)=0\vspace*{-7pt}\]
Si $f(x)\in k(x)$ est une fraction rationnelle quelconque, non constante, $y=f(x)$ est une courbe de genre 0 et le revêtement de degré 2 $(x,y)\mapsto x^2$ est ramifié en 0 et $\infty$.

On cherche ensuite un revêtement de $Z$ de degré $l$, non ramifié en $0,\infty$ mais en deux autres points, dont la ramification est nécessairement d'ordre $l$. On vérifie qu'en choisissant $P(x)\in k[x]$ un polynôme de degré 2, la courbe $Y$ définie par\vspace*{-9pt} 
\[t^l=y\quad \textrm{et}\quad y=\frac{P(x)}{P(-x)},\vspace*{-9pt}\]
possède les propriétés souhaitées : en effet, le recouvrement $Y\rightarrow Z$ donné par $(x,y,t)\mapsto(x,y)$ est ramifié d'ordre $l$ en les 4 racines des polynômes $P(x)$ et $P(-x)$, qui sont ensuite identifiées deux par deux par $(x,y)\mapsto x^2$.

Il est clair que l'on peut choisir $P$ unitaire, et que par un changement de variable homographique $x\mapsto \lambda x$, on peut choisir $P(x)=x^2+x+a$, $P$ ne pouvant pas être un polynôme pair.
\begin{prop}
Soit $Y$ la courbe définie par l'équation\vspace*{-8pt} \[t^l=\frac{x^2+x+a}{x^2-x+a}.\vspace*{-7pt}\]
Alors, le revêtement $Y\rightarrow\Pp^1$, $(x,y)\mapsto x^2$ est de degré $2l$, de type $(0,0,1,1)$ et de groupe de Galois le groupe diédral $G_{l,2}$.
\end{prop}
\begin{proof}
Il ne reste qu'à vérifier que le groupe de Galois du revêtement est bien $G_{l,2}$. En effet, on trouve les automorphismes $\sigma:(x,t)\mapsto(x,\zeta_lt)$ et $\alpha:(x,t)\mapsto\big(-x,\tfrac1t\big)$ d'ordres respectifs $l$ et 2, tels que\vspace*{-10pt} 
\[\alpha\sigma\alpha^{-1}:(x,t)\mapsto(x,\zeta_l^{-1}t)=\sigma^{-1}(x,t).\qedhere\]
\end{proof}
\vspace*{-4pt}
Grâce à la construction de J. Ellenberg présentée dans la section \hyperref[sec:ellen1]{précédente}, ceci fournit donc une famille à un paramètre, $a$, de courbes à multiplication réelle par $\Q(\zeta_l^+)$, dont on peut mener tous les calculs explicitement. 

En effet, la courbe quotient $Y/\langle\alpha\rangle$ est obtenue  en calculant le polynôme minimal $\Pi(z,x)$ de $z_t:=t+\tfrac1t$, qui doit nécessairement être un polynôme invariant par $x\mapsto-x$ : une équation de $X$ est donnée par le polynôme $Q(z,w)$ tel que $\Pi(z,x)=Q(z,x^2)$. On peut voir $Q$ comme le reste de la division euclidienne en $x$ de $\Pi$ par $w^2-x$.

On retrouve en fait la même famille que dans~\cite{top}. En effet, la courbe $Y$ possède un autre modèle qui rejoint celui des courbes de cet article.

\begin{prop}\label{top:fam2}
La courbe $Y$ est une courbe hyperelliptique de genre $l-1$. Elle peut être donnée, à torsion près, par une équation de type\vspace*{-2pt}
\[y^2=z^{2l}+bz^l+1.\vspace*{-2pt}\]
La courbe $X$ est elle aussi hyperelliptique, donnée par l'équation\vspace*{-3pt}
\[\mu^2=(w+2)(wg_l(w^2-2)+b)\vspace*{-2pt}\]
où $g_l$ est le polynôme minimal de $-\zeta_l-\zeta_l^{-1}$.
\end{prop}
\begin{proof}
Il s'agit simplement d'un changement de variables,\vspace*{-3pt}
\[\Big\{\begin{array}{l@{}l}
\\[-22pt]
z&\,=\,t\\[-3pt]
y'&\,=\,2(t^l-1)x-(t^l+1),
\end{array}\vspace*{-2pt}
\]
qui aboutit au modèle hyperelliptique $y'^2=(1-4a)(z^{2l}+2\tfrac{1+4a}{1-4a}z^l+1)$ et, quitte à prendre une racine carrée de $1-4a$ et à poser $b=\tfrac{1+4a}{1-4a}$ comme nouveau paramètre, on obtient la courbe hyperelliptique de genre $l-1$\vspace*{-1pt}
\[y^2=z^{2l}+bz^l+1.\vspace*{-1pt}\]
\noindent
Le groupe de Galois sur ce modèle est réalisé par les éléments notés comme avant\vspace*{-3pt} 
\begin{align*}
\sigma:(z,y)&\mapsto(\zeta_lz,y)\\
\alpha:(z,y)&\mapsto\big(\tfrac1{z},\tfrac{y}{z^l}\big).\\[-20pt]
\end{align*}
\noindent
Le reste de la démonstration est identique à~\cite{top} : on pose $\omega=z+\tfrac1{z}$ et on a\vspace*{-2pt}
\begin{align*}
z^{2l}+1&=z^l\big(z+\frac1z\big)g_l\big(z^2+\frac1{z^2}\big)\\
\omega+2&=\frac{(z+1)^2}z
\end{align*}
qui permet d'écrire $y^2=z^l(b+\omega g_l(\omega^2-2))$ puis\vspace*{-2pt} \[z^l=\Big(\frac{z^{\frac{p+1}2}}{z+1}\Big)^2(\omega+2)\]
pour en tirer finalement, après avoir posé $\mu=\tfrac{y(z+1)}{z^{\frac{l+1}2}}$ que
\[\mu^2=(\omega+2)\big(b+\omega g_l(\omega^2-2)\big)\]
et que les fonctions $\mu(z,y)$ et $\omega(z)$ sont invariantes par $\alpha$, ce qui assure que l'on a trouvé un modèle de la courbe quotient $X=Y/\langle\alpha\rangle$.
\end{proof}

On peut vérifier simplement, comme dans~\cite{top}, que la courbe quotient $X=Y/\langle\alpha\rangle$ est bien à multiplication réelle par $\Q(\zeta_l^+)$, en vérifiant\footnote{On\,note\,toujours\,$[\sigma]$\,l'endomorphisme\,de\,la jacobienne,\,issu\,d'un\,automorphisme\,$\sigma$\,de\,la\,courbe.}  que $[\sigma]+[\sigma^{-1}]$ est bien un endomorphisme de la jacobienne du quotient : cela suffit puisque l'on est alors assuré d'avoir le corps totalement réel  $\Q(\zeta_l^+)$ dans $\End_\Q(\Jac(X))$ d'une part et que d'autre part, l'expression explicite de la courbe hyperelliptique $X$ assure qu'elle est de genre $\tfrac{l-1}2$, étant définie par un polynôme de degré $l+1$.

\subsection{Type (0,1,1,1,1)}\label{type:01111}
On cherche, comme précédemment, un revêtement de $\Pp^1$ ramifié en 4 points d'ordre 2. On trouve ici les courbes de genre 1 données par une équation de \WS, famille à un paramètre, avec le revêtement qui \og oublie \fg{} les ordonnées du modèle de \WS. En munissant $E$ de son point à l'infini, on en fait une courbe elliptique sur laquelle on veut un revêtement ramifié en seulement un point, de groupe $\Zz l$. Il faut, de plus, que le groupe de Galois du revêtement total soit le groupe diédral. On cherche alors une fonction sur $E$ qui, de façon similaire au cas précédent, soit changée en son opposée par l'involution hyperelliptique.

\begin{prop}
Il existe une famille rationnelle sur $\Q$, à deux paramètres, de courbes à multiplication réelle par $\Q(\zeta_l^+)$, issue du type $(0,1,1,1,1)$.
\end{prop}

\begin{proof}
Commençons par considérer un point générique $P\in E$ et posons $Q=lP$. Le diviseur sur $E$, de degré 0,\[D=(Q)-(-Q)+l((P)-(-P))\]  est   
principal car $Q-(-Q)+lP-l(-P)=O_E$. C'est le diviseur d'une fonction dont on souhaite, comme avant, qu'elle soit changée en son opposée par l'involution elliptique $\imath$. Pour cela, en concentrant les pôles à l'infini, on écrit $D=D'+\imath(D')$ avec
\[D'=(-Q)+l(P)-(l+1)O_E.\]
Pour les mêmes raisons que $D$, le diviseur $D'$ est principal, et soit $\phi(x,y)=a(x)+yb(x)$ une fonction de diviseur $D'$. Alors, le diviseur de $\tfrac{\phi(x,y)}{\phi(x,-y)}$ est $D$. On considère ensuite le revêtement de $E$ par la courbe $Y$ définie par l'équation de $E$ ainsi que 
\[t^l=\frac{\phi(x,y)}{\phi(x,-y)}.\]
Le revêtement est donné par l'application $(x,y,t)\mapsto(x,y)$ qui n'est ramifié qu'en $Q$ et $-Q$ car la ramification en $P$ et $-P$ est tuée par le facteur $l$ dans le diviseur de la fonction $\phi$. Ainsi, le revêtement cherché de $Y$ sur $\Pp^1$ est donné par $(x,y,t)\mapsto x$.

Il reste à déterminer le groupe de Galois. On a déjà l'involution hyperelliptique, qui \og monte \fg{}, par construction, sur $Y$ ainsi que les racines $l$-ièmes agissant $t$ :
\begin{align*}
\alpha:(x,y,t)&\mapsto(x,-y,\tfrac 1t)\\
\sigma:(x,y,t)&\mapsto(x,y,\zeta_l y)
\end{align*} 
Comme précédemment, on vérifie sans difficulté que $\alpha\sigma\alpha^{-1}=\sigma^{-1}$ si bien que le groupe diédral est bien le groupe de Galois de ce revêtement.

Le choix de $E$ et de $P\in E$ point générique constitue les deux paramètres de la familles recherchée. Sa rationalité est donnée dans l'algorithme~\ref{algo:mult:reelle2} ci-après. 
\end{proof}

On peut, comme dans le cas précédent, mener les calculs pour obtenir des familles à 2 paramètres. Du moins, on a un algorithme facile permettant, une fois donné $l$, d'écrire une équation de $X=Y/\langle\alpha\rangle$. Pour cela, on part d'une équation générale de courbe elliptique $E$ de la forme $y^2=x^3+ux^2+vx+w^2$, possédant de plus le point rationnel \og générique \fg{}  $P=(0,w)$. On obtient les deux paramètres de la famille par homogénéité de l'équation $E$ ; le changement de paramètres rationnels
\[%
u=\tau s,\quad v=\tau^2r,\quad w^2=\tau^3r\quad\textrm{ et}\quad x=\tau x'%
\]
élimine la variable $\tau$ et donne une famille de courbes en $x',z$ à 2 paramètres, $r,s$.

\begin{algo}[H]
\caption{Courbes à multiplication réelle par $\Q(\zeta_l^+)$ (I)}
\label{algo:mult:reelle}
\noindent\textbf{Entrée} : $l$
\begin{algorithmic}[1]
\State{$E(x,y)\leftarrow y^2=x^3+ux^2+vx+w^2$}
\State{$P\leftarrow(0,v)$ et $Q\leftarrow lP$}
\State{\textbf{Calculer}, à l'aide de \RR, une fonction $\phi(x,y)$ telle que\vspace*{-10pt} 
\[\div(\phi)=(-Q)+l(P)-(l+1)O_E.\]}\vspace*{-21pt}
\State{$\Pi(x,y,z)\leftarrow\Res_t(\phi(x,-y)t^l-\phi(x,y),t^2+1-tz)$}\label{algo:mult:reelle:ligne5}
\State{$X(x,z)\leftarrow \Pi(x,y,z)\mod\!_y\  E(x,y)$}\label{algo:mult:reelle:ligne6}
\end{algorithmic}
\noindent\textbf{Sortie} $X$.
\end{algo}
Avant de donner des exemples, il faut faire quelques remarques sur cet algorithme. Comme précédemment, la ligne~\ref{algo:mult:reelle:ligne5} permet de calculer le polynôme minimal de $z_t:=t+\tfrac1t$ à l'aide d'un calcul classique d'élimination par résultant. Enfin, il est à préciser que dans la ligne~\ref{algo:mult:reelle:ligne6}, ce résultant, modulo l'équation de $E$, ne dépend plus de $y$, puisque $y\mapsto-y$ change $t$ en $\tfrac1t$, et laisse donc $z_t$ invariant.
\begin{expl}[Courbes à multiplication réelle par $\Q(\zeta_5^+)$ et $\Q(\zeta_7^+)$] Voici deux exemples de familles à deux paramètres fournies par l'algorithme~\ref{algo:mult:reelle}.
\begin{itemize0}
\item[\qquad\bf$\mathbold{l=}$\,5]\label{q5plus} On a $\Q(\zeta_5^+)=\Q(\sqrt 5)$ et $X$ est une courbe de genre 2, hyperelliptique. L'algorithme ci-dessus et le changement de variables $x\mapsto\tfrac1x$ fournissent l'équation\vspace*{-10pt}
\[z(z^{4}-5z^{2}+5)=\frac1{(d_1x+d_0)}\sum\limits_{i=0}^6c_ix^i,
\]
où on a l'identité polynomiale $z^4-5z^2+5=g_5(z^2-2)$, avec $g_5$  défini en~\ref{top:fam2}. L'équation ci-dessus est à variables séparées de la forme $g(z)=h(x)$, avec la relation sur les différentielles\vspace*{-5pt}
\[g'(z)\D z=h'(x)\D x.\]
On peut alors calculer une base de différentielles holomorphes $(\omega_1\D x,\omega_2\D x)$ puis poser un nouveau paramètre $X=\tfrac{\omega_2}{\omega_1}$, qui nous permet de calculer $\D X=\bigl(\tfrac{\partial X}{\partial x}+\tfrac{\partial X}{\partial z}\tfrac{h'(x)}{g'(z)}\bigr)\D x$, et de poser $Y=\tfrac{\D X}{\omega_1\D x}$. On écrit alors\vspace*{-15pt}\[Y^2=\sum_{i=0}^6a_iX^i,\]
on réduit modulo l'équation de départ, puis on résout les équations linéaires en les $a_i$. Il est pratique de poser $s=r+t$ et on obtient une famille à deux paramètres sous forme hyperelliptique donnée par les coefficients $a_i$,
dont les coefficients sont donnés par
{\footnotesize
\begin{align*}
a_{{0}}&=\frac{{100{t}^{2}{r}^{3}+5{r}^{3}-36t{r}
^{3}-8{t}^{3}{r}^{2}-40{t}^{5}{r}^{2}+36{r}^{2}{t}^{4}-128{t}^
{3}{r}^{3}-20{t}^{6}r-12{t}^{8}r+64{r}^{3}{t}^{4}+48{t}^{7}r-8
{t}^{9}}}{r},\\a_{{1}}&=20(-5{t}^{8}-8{t}^{7}r+28{t}^{6}r-12r{t}^{5}-24{r}^{2}{t}^{4}+
16{t}^{3}{r}^{3}+24{t}^{3}{r}^{2}-24{t}^{2}{r}^{3}-6{t}^{2}{r}
^{2}+12t{r}^{3}-2{r}^{3}),\\a_{{2}}&=a_4=0,
\\a_{{3}}&=10( {r}^{2}-4t{r}^{2}+4{t}^{3}r+4{t}^{2}{r}^{2
}-10r{t}^{4}+2{t}^{6}+4r{t}^{5}) {r}^{2},\\a_{{5}
}&=4(6{r}^{5}{t}^{2}-{t}^{4}{r}^{4}-4{r}^{5}{t}^{3}-2t{r}^{5}),\\a_{{6}}&=
r^6-4r^6t+4r^6t^2
\end{align*}
}%
Enfin, comme on a  deux paramètres en genre 2, cette famille contient donc, à rationalité près, toutes les courbes de genre 2 à multiplication réelle par $\Q(\zeta_5^+)$, qui sont, par exemple, données dans~\cite{mestrehyp} sous une forme hyperelliptique plus simple.
\item[\qquad\bf$\mathbold{l=}$\,7] Cette fois, le corps $\Q(\zeta_7^+)=\Q(\cos(\tfrac{2\pi}7))$ est de degré 3 et l'on obtient une famille à deux paramètres de courbes de genre 3 à multiplication réelle par ce corps. Elle est donnée, de façon similaire, par une équation de la forme
\[z(z^6-7z^4+14z^2-7)=\frac1{(d_1x+d_0)}\sum\limits_{i=0}^8c_ix^i.\]
Cette famille est différente des familles exposées dans~\cite{mestrehyp} car on peut déterminer une équation quartique
 à deux paramètres. Comme ci-dessus, le changement de variable $s=r+t$ diminue fortement la taille de l'équation et permet de calculer une base de différentielles, par exemple avec la commande \texttt{differentials} du paquet \texttt{algcurves} de Maple~\cite{maple}. On obtient l'équation générale en $U$ et $V$ avec deux paramètres $r$ et $t$.

{\relpenalty=0000
\binoppenalty=0000
\footnotesize
\noindent$- ( 2tr-r-{t}^{3} )  ( -{r}^{3}-40{t}^{2}{r}^{3}+
10t{r}^{3}-80{r}^{3}{t}^{4}+80{t}^{3}{r}^{3}+32{r}^{3}{t}^{5}+
16{r}^{2}{t}^{7}-56{r}^{2}{t}^{6}+64{r}^{2}{t}^{5}-30{r}^{2}{t
}^{4}+5{r}^{2}{t}^{3}+12r{t}^{8}-16r{t}^{7}+5r{t}^{6}+{t}^{9}
 ) +2r ( 2t-1 ) ^{2} ( {r}^{2}-4t{r}^{2}+4
{t}^{3}r+4{r}^{2}{t}^{2}-10r{t}^{4}+2{t}^{6}+4r{t}^{5}
 ) U+{r}^{2} ( 2t-1 ) ^{4}U^{2}-r ( 2t-1
 ) ^{2}U^{3}+r ( 2t-1 ) ^{2} ( -2{r}^{3}-6
{r}^{2}{t}^{2}+12t{r}^{3}+24{r}^{2}{t}^{3}-24{t}^{2}{r}^{3}-12
r{t}^{5}-24{r}^{2}{t}^{4}+16{t}^{3}{r}^{3}+28r{t}^{6}-5{t}^{
8}-8r{t}^{7} ) V-2{t}^{2} ( {t}^{3}+2r-6tr+4{t}^{
2}r ) ^{2}UV-t ( {t}^{3}+2r-6tr+4{t}^{2}r ) 
 ( {r}^{2}-4t{r}^{2}+4{t}^{3}r+4{r}^{2}{t}^{2}-10r{t}^{4}
+2{t}^{6}+4r{t}^{5} ) V^{2}-2rt ( 2t-1 ) ^{
2} ( {t}^{3}+2r-6tr+4{t}^{2}r ) UV^{2}+t ( {t}
^{3}+2r-6tr+4{t}^{2}r ) U^{2}V^{2}+2r ( 2t-1
 ) ^{2} ( {r}^{2}-4t{r}^{2}+4{t}^{3}r+4{r}^{2}{t}^{2}
-10r{t}^{4}+2{t}^{6}+4r{t}^{5} ) V^{3}+ ( {r}^{2}-4
t{r}^{2}+4{t}^{3}r+4{r}^{2}{t}^{2}-10r{t}^{4}+2{t}^{6}+4r{
t}^{5} ) UV^{3}+{t}^{2} ( {t}^{3}+2r-6tr+4{t}^{2}r
 ) ^{2}V^{4}.$
}%

Finissons cet exemple par une courbe spéciale dans cette famille, obtenue en fixant au début $u=v=0$ ; on obtient dans un premier temps la courbe
\[x^6z(z^6-7z^4+14z^2-7)=2x^6+4x^3+1\]
qui a un modèle quartique donné par l'équation 
\[v^3(2u-1)-uf(u)=0,\]
où $f(u)=u^3+2u^2-u-1$ est le polynôme minimal de $(\zeta_7+\zeta_7^{-1})^{-1}$. Notons que sur les deux modèles, il est facile de voir que l'anneau des endomorphismes de cette courbe contient $\Q(\zeta_3)=\Q(i\sqrt3)$, par $(x,z)\mapsto(\zeta_3x,z)$, ou $(u,v)\mapsto(\zeta_3u,v)$ sur le modèle quartique. Ceux-ci proviennent de la courbe elliptique $y^2=x^3+1$. Comme cet automorphisme sur $Y$ commute avec $\sigma$ et $\alpha$, on en déduit que $X$ est à multiplication complexe par $\Q(\zeta_7^+,i\sqrt3)$.
\end{itemize0}

\end{expl} 
\subsection{Type (1,1,1,1,1,1)}\label{type:111111}
Un tel revêtement est donné par une courbe de genre 2, disons \[y^2=x(x-1)(x-\lambda_1)(x-\lambda_2)(x-\lambda_3),\] que l'on note $H$. Le revêtement de degré 2 ramifié en les  6 points de \WS est bien-entendu $(x,y)\mapsto x$.

\begin{prop}
Il existe une famille de courbes à trois paramètres, \og explicite \fg{}, à multiplication réelle par $\Q(\zeta_l^+)$.
\end{prop}

\begin{proof}
On cherche un revêtement non ramifié de la courbe hyperelliptique $H$, sur lequel agit le groupe diédral $G_{l,2}$.

Pour cela, on considère $P\in\Jac(H)$ de $l$\ti torsion, défini dans la clôture algébrique $k$ de $\Q(\lambda_1,\lambda_2,\lambda_3)$. Ainsi, par définition le diviseur de $lP$, en représentation de Mumford, est un diviseur principal : c'est le diviseur d'une fonction $\phi(x,y)$. 
Comme dans le cas précédent, il nous suffit ensuite de considérer la courbe $Y$ donnée par les équations
\[
(Y)\ \Bigg\{\begin{array}{l@{\,=\,}l}
y^2&x(x-1)(x-\lambda_1)(x-\lambda_2)(x-\lambda_3)\\
t^l&\frac{\phi(x,y)}{\phi(x,-y)}
\end{array}
\]
dont le recouvrement de $H$ par $(x,y,t)\mapsto(x,y)$ est non ramifié, de degré $l$.

On a de manière analogue aux deux cas précédents, les morphismes
\begin{align*}
\alpha:(x,y,t)&\mapsto(x,-y,\tfrac 1t)\\
\sigma:(x,y,t)&\mapsto(x,y,\zeta_l t)
\end{align*}
qui engendrent le groupe diédral $G_{l,2}$. Encore une fois, pour donner une équation de la courbe quotient $X=Y/\langle\alpha\rangle$, on commence par calculer le polynôme minimal en $z$ de $z_t:=t+\tfrac1t$ dans $k(x,y)$. Le résultat est invariant par $y\mapsto-y$ et ainsi, modulo l'équation de $H$, on obtient un polynôme en $x$ et $z$ qui est une équation plane de la courbe $X$.
\end{proof}

Cette démonstration donne, comme dans le cas précédent, un algorithme de calcul d'une équation plane de $X$, en tout point similaire à l'algorithme~\ref{algo:mult:reelle}, à la différence que l'on doit trouver un élément $P\in\Jac(H)$ de $l$\ti torsion.

Pour cela, on peut faire comme dans le cas des courbes elliptiques, où, partant d'un point \og générique \fg{} $(x,y)$, on applique successivement la loi d'addition. On aboutit en fait aux polynômes de division (que l'on peut calculer plus efficacement) dont les racines sont les points de $l$\ti torsion.  En genre 2, on peut faire la même chose, en utilisant la représentation de Mumford des points de  la jacobienne. Ainsi, on choisit de représenter un point $P$ de la jacobienne de la courbe hyperelliptique $H$, par deux points $P_1,P_2$ de la courbe $H$, et de l'identifier au diviseur 
\begin{equation}\label{div:jac}
\Div(P):=(P_1)+(P_2)-D_\infty,
\end{equation} où $D_\infty=2(P_\infty)$, s'il n'y a qu'un point à l'infini, ou $D_\infty=(P_{\infty,1})+(P_{\infty,2})$ sinon. Ensuite, on applique les lois d'addition sur la jacobienne, de façon similaire au genre~1. Néanmoins, même si tous les calculs sont effectifs, cet algorithme reste théorique au sens où la puissance de calcul nécessaire reste, à l'heure actuelle, insuffisante.

\paragraph{Familles de courbes hyperelliptiques à deux paramètres}

Dans le cas où la jacobienne de la courbe de genre 2 n'est pas simple et que l'on a un morphisme vers une courbe elliptique, J.-F. Mestre, dans~\cite{mestrehyp}, donne une famille à deux paramètres de courbes hyperelliptiques à multiplication réelle par $\Q(\zeta_l^+)$. Celle-ci repose aussi sur un point de la courbe elliptique. Ainsi, quand $X_1(l)$ est de genre 0, cela fournit des familles rationnelles, explicitées dans~\cite{mestrehyp}, pour $l=5,7$ par exemple.

\paragraph{Courbes modulaires $X_0(n)$} Dans le cas où la jacobienne de $H$ est simple, on trouve d'autres courbes que celles fournies plus haut. On peut par exemple donner des exemples issus des courbes modulaires $X_0(n)$ qui sont des courbes (hyperelliptiques) de genre 2. On sait --- voir par exemple~\cite{ogg} --- qu'il n'existe qu'un nombre fini de courbes modulaires qui sont hyperelliptiques, qui plus est de genre 2. Elles sont données pour les entiers $n\in\{22,23,26,29,31,37\}$. On considère leur modèle de \WS, avec deux points à l'infini. Parmi ces entiers $n$, il y a soit des nombres premiers, soit des composés de deux nombres premiers. Toujours dans~\cite{ogg}, on a des points rationnels sur les jacobiennes de $X_0(n)$ dont on connaît l'ordre. Ainsi, pour $n$ premier, le point\footnote{Si l'on se tient à la formule~\ref{div:jac} que l'on a choisie pour représenter un point de la jacobienne, il s'agit de $2(P_{\infty,1})-(P_{\infty,1})-(P_{\infty,2})$, qui est bien le même.} $(P_{\infty,1})-(P_{\infty,2})$ est d'ordre le numérateur de $\tfrac{n-1}{12}$, et dans le cas où $n=pq$, on a des points d'ordre les numérateurs des $\tfrac{(p-1)(q-1)}{24}$, $\tfrac{(p+1)(q-1)}{24}$ et $\tfrac{(p-1)(q+1)}{24}$. On peut résumer cela dans la table~\ref{xon} suivante, où les équations ont été obtenues par Fricke~\cite{fricke}

\begin{table}[h!]
\centering
\begin{tabular}{|c|c|c|c|}
\hline
$n$&Équation de $X_0(n)$&\!\!Points\,rationnels\!\!&\!Ordre $l$\!
\\\hline&&&\\[-1.3em]\hline
&&$(P_{\infty,1})-(P_{\infty,2})$&5\\
22&$y^2\!=\!(x^3\!+\!8x^2\!+\!16x\!+\!16)(x^3\!+\!4x^2\!+\!8x\!+\!4)$&$(0,8)-(P_{\infty,2})$&5\\
&&\!$(0,-8)\!-\!(P_{\infty,2})$\!&5\\
\hline
23&$y^2\!=\!x^6\!-\!14x^5\!+\!57x^4\!-\!106x^3\!+\!90x^2\!-\!16x\!-\!19$&$(P_{\infty,1})-(P_{\infty,2})$&11\\
\hline
26&$y^2=x^6-8x^5+8x^4-18x^3+8x^2-8x+1$&$(P_{\infty,1})-(P_{\infty,2})$&21\\
\hline
29&$y^2=x^6-4x^5-12x^4+2x^3+8x^2+8x-7$&$(P_{\infty,1})-(P_{\infty,2})$&7\\\hline
31&$y^2\!=\!x^6\!-\!14x^5\!+\!61x^4\!-\!106x^3\!+\!66x^2\!-\!8x\!-\!3$&$(P_{\infty,1})-(P_{\infty,2})$&5\\\hline
37&$y^2=x^6-9x^4-11x^2+37$&$(P_{\infty,1})-(P_{\infty,2})$&3\\\hline
\end{tabular}
\caption{Points rationnels de courbes modulaires hyperelliptiques de genre 2}
\label{xon}
\end{table}
Dans toutes les situations, on voit que le point $(P_{\infty,1})-(P_{\infty,2})$ des jacobiennes  permet de construire des courbes à multiplication réelle par l'un des $\Q(\zeta_5^+)$, $\Q(\zeta_7^+)$ ou $\Q(\zeta_{11}^+)$, les points d'ordre 3 ne donnant rien ici car $\Q(\zeta_3^+)=\Q$.

La forme de ce diviseur, très particulière, permet de construire très facilement les courbes $Y$ et $X$ en donnant des formes isomorphes, mais plus faciles à calculer. En effet, le diviseur principal $l((P_{\infty,1})-(P_{\infty,2}))$ est le diviseur d'une fonction de la forme $P(x)+yQ(x)$ telle que sa norme $P^2(x)-H(x)Q^2(x)$ soit une constante, où l'on a noté $y^2=H(x)$ la fonction définissant la courbe modulaire hyperelliptique.

Cela a deux avantages. D'une part, on peut calculer assez facilement une telle fonction, sans utiliser les méthodes liées au théorème de \RR. Supposons $H(x)$ unitaire et écrivons
\[x^6H(\tfrac1x)=\frac{\big(x^lP(\tfrac1x)\big)^2-cx^{2l}}{\big(x^{l-3}Q(\tfrac1x)\big)^2},\]
où $c$ est la constante telle que $P^2(x)-H(x)Q^2(x)=c$. Par développements limités, on a alors
\[\sqrt{x^6H(\tfrac1x)}=\frac{x^lP(\tfrac1x)}{x^{l-3}Q(\tfrac1x)}-\frac c2x^{2l}+o(x^{2l}),\]
si bien que les approximants de Padé de $\sqrt{x^6H(\tfrac1x)}$ d'ordre $(l,l-3)$ sont $x^lP(\tfrac1x)$ et $x^{l-3}Q(\tfrac1x)$, ce qui s'obtient aisément, par exemple par la résolution d'un système linéaire. 

D'autre part, par le théorème 90 de Hilbert, à la constante $k=P(x)^2-H(x)Q(x)^2$ près, la fonction obtenue s'écrit déjà sous la forme $\tfrac{\phi(x,y)}{\phi(x,-y)}$. Ainsi, quitte à multiplier $\phi$ par une constante telle que sa norme soit une puissance $l$-ième, que l'on note $\kappa^l$, on a pour $Y$ les équations $\{y^2=H(x)\textrm{ et }t^l=\phi(x,y)\}$, avec comme automorphismes

\begin{align*}
\alpha:(x,y,t)&\mapsto(x,-y,\tfrac \kappa t)\\
\sigma:(x,y,t)&\mapsto(x,y,\zeta_l t)
\end{align*}
et $z_t:=t+\tfrac\kappa t$, l'élément invariant permettant de calculer une équation de $X$. 

\begin{expl}[Courbes à multiplication réelles par $\Q(\zeta_l^+)$ issues des courbes modulaires $X_0(n)$]On finit ce paragraphe par donner des exemples pour $l=5$ et $7$, les calculs pour les différents $l$ étant d'une difficulté quasi-égale pour la machine.

\begin{itemize0}
\item[\qquad\bf$\mathbold{l=}$\,5] On utilise la courbe modulaire $X_0(31)$ et les approximants de Padé donnent
\[\phi(x,y)=x^5-19x^4+125x^3-328x^2+280x-13+y(x^2-12x+35)\]
de diviseur $5((P_{\infty,1})-(P_{\infty,2}))$ et de norme $62^2$ permettant de prendre $\kappa=1$ avec $\frac\phi{62}$. On obtient une équation de $X,$ $z g_5(z^2-2) =\tfrac1{31}({x}^{5}-19{x}^{4}+125{x}^{3}-328{x}^{2}+280x-13)$ et sous forme hyperelliptique $y^2 -\tfrac1{31}y = x^6 + x^5 - \tfrac5{31}x^3 + \tfrac4{31^2}x$, 
qui a mauvaise réduction en $5$ et en $31$.
\item[\qquad\bf$\mathbold{l=}$\,7] Ici, on choisit la courbe modulaire $X_0(29)$. La fonction $\phi(x,y)$ déterminée par les approximants de Padé
\[x^7\!-13x^6+45x^5+25x^4\!-269x^3+29x^2+300x+166+(x^4\!-11x^3+31x^2+14x-100)y\]
a pour norme $29\cdot58^2$, ce qui impose de prendre $\kappa=29\neq1$ et $\tilde\phi=\tfrac{29^2}{2}\phi$. Comme $X$ est de genre 3 non-hyperelliptique, on peut la mettre sous forme quartique,
{\footnotesize
\begin{align*}841u^4&-1682u^3v+116uv^3+29v^4+841u^3-232u^2v+58uv^2+58v^3-87u^2+58uv-13v^2-58u-17v+4
\end{align*}
}%
qui a mauvaise réduction en $7$ et en $29$.
\end{itemize0}
\end{expl}

\paragraph{D'autres courbes rationnelles ayant des diviseurs rationnels d'ordre $l$}

Outre les courbes modulaires, il existe beaucoup d'autres courbes qui permettent des constructions similaires, aboutissant à des courbes à multiplication réelle par $\Q(\zeta_l^+)$ qui, \emph{a priori}, n'appartiennent pas aux différentes familles que l'on a déjà vues et mentionnées. On peut notamment citer les familles à 1 paramètre de courbes de genre~2 possédant des points d'ordre 15 ou 21 exposées par F. Leprévost dans~\cite{leprevost91} ou bien encore d'autres courbes présentes dans~\cite{leprevost95} du même auteur. 
On peut aussi, comme pour les courbes modulaires, chercher des courbes hyperelliptiques de genre 2 dont le point de la jacobienne $P_{\infty,1}-P_{\infty,2}$ est d'ordre fini $l$. Les outils algorithmiques déjà rencontrés permettent de démontrer les deux propositions suivantes.
\begin{prop}
Les courbes hyperelliptiques de genre $2$, de la famille à deux paramètres $s$ et $t$, d'équation\vspace*{-8pt}
\[y^2=s(16t-27+4s)x^6+2s(-16t+27+s)x^5+10stx^3+2t^2x+t^2,\vspace*{-5pt}\]
sont à multiplication réelle par $\Q(\zeta_5^+)$.
\end{prop}
\begin{proof}
Il s'agit une fois encore de la totalité, à rationalité près, des courbes à multiplication réelle par  $\Q(\zeta_5^+)$. Soit $y^2=x^6+ax^4+bx^3+cx^2+dx+e$ une courbe hyperelliptique de genre 2 générale. On  calcule, sous représentation de Mumford, les diviseurs $2(P_{\infty,1}-P_{\infty,2})$ et $3(P_{\infty,1}-P_{\infty,2})$ dont on égale la première coordonnée (unitaire). Cela nous donne deux équations, le coefficient en $x$ et le coefficient constant, dont on prend le résultant en $e$. Un des facteurs est de genre 0 que l'on paramètre. On obtient une famille de courbes de genre 2 dont le point $(P_{\infty,1}-P_{\infty,2})$ est d'ordre~5.

{\footnotesize
\relpenalty=0000
\binoppenalty=0000
\noindent$y^2=82944x^6\!+\!(13824t\!+\!3456s\!-\!31104 ) x^{4}\!+\!(1728s\!-\!3456t\!+\!5184)x^3\!+\!(864ts\!-\!648s\!+\!36s^2\!+\!576{t}^{2}\!-\!2592t\!+\!2916)x^2\!+\!(\!-\!972\!+\!216ts\!-\!288t^2\!+\!1080t\!-\!216s\!+\!36s^2)x\!+\!81\!+\!54s\!-\!108t\!+\!36t^2\!-\!108ts\!+\!9s^2\!+\!48t^2s$.}

Il ne reste plus qu'à appliquer la méthode décrite ci-dessus des approximants de Padé pour trouver une fonction $\phi$. Ensuite, on mène les mêmes calculs que ceux expliqués dans l'exemple~\ref{q5plus} pour trouver le modèle hyperelliptique donné.
\end{proof}

\begin{prop}
Les courbes quartiques suivantes sont à multiplication réelle par $\Q(\zeta_7^+)$. Elles appartiennent à deux familles, la première à un paramètre $s$,\vspace*{-8pt}
\[2v + u^3 +(u+1)^2+s\bigl((u^2+v)^2 - v (u+v)(2u^2-uv + 2v)\bigr)\vspace*{-10pt}\]
et la seconde en deux paramètres $s$ et $t$, \vspace*{3pt}

{
\relpenalty=0000
\binoppenalty=0000
\noindent $- ( s+t ) ^{2}+2 ( s+t ) sv+ ( -{s}^{2}+3
{t}^{2}-t ) v^{2}+ ( 6{t}^{2}-2t-2{s}^{2}
 ) v^{3}+2 ( s+t ) ^{2}u+ ( 6{t}^{2}-2t-
2{s}^{2} ) uv^{2}+ ( -{s}^{2}+3{t}^{2}-t ) uv
^{3}- ( s+t )  ( s-t ) u^{2}+ ( 2t+2{
s}^{2}-6{t}^{2} ) u^{2}v+ ( t-3{t}^{2}+{s}^{2}
 ) u^{2}v^{2}+ ( s+t )  ( s-t ) u^{3
}+ ( 2t+2{s}^{2}-6{t}^{2} ) u^{3}v+ ( -{s}^{2}
+3{t}^{2}-t ) u^{4}.$\vspace*{-5pt}
}
\end{prop}
\begin{proof}
La situation est plus délicate ici. Pour pouvoir obtenir une condition de genre~0, on commence par considérer une courbe hyperelliptique d'équation $y^2=f(x)=x^6+ax^4+bx^3+cx^2+dx$, puis on calcule les approximants de Padé de $\sqrt{\tfrac1{x^6}f(\tfrac1x)}$ d'ordre $(7,4)$. On les note $\tfrac1{x^7}P(\tfrac1x)$ et $\tfrac1{x^4}Q(\tfrac1x)$ si bien que $P(x)^2-f(x)Q(x)^2$ est un polynôme de degré 2. Le résultant des coefficients en $x$ et $x^2$ relativement à $a$ possède un facteur de genre 0 que l'on paramètre : on obtient la courbe hyperelliptique à un paramètre $t$,\vspace*{-5pt}
\[y^2={x}^{6}+ \tfrac{2t-1}2{x}^{4}+t{x}^{3}+ \tfrac{4t^2-12t+1}{16} {x}^{2}+\tfrac{{t}^{2}}2x\]
dont on peut vérifier directement, par un calcul d'approximation de Padé similaire, que le point $(P_{\infty,1}-P_{\infty,2})$ est bien d'ordre 7.

Il ne reste plus qu'à calculer l'équation de $X$ comme dans le paragraphe précédent, puis à déterminer une équation quartique, en utilisant une base de différentielles. On aboutit à la première famille.

Pour la seconde famille, on part cette fois d'une équation hyperelliptique générale de degré 6, que l'on peut toujours ramener à la forme $y^2=x^6+ax^5+ax^4+bx^3+cx^2+dx$. En effectuant des calculs similaires d'approximation de Padé, 
on trouve la famille de courbes à deux paramètres $s,t$,

{\footnotesize
\begin{align*}y^2={x}^{6}&+ ( 2t+s ) {x}^{5}+ ( 2t+s ) {x}^{4}+
 \big( \tfrac12{s}^{2}-t{s}^{2}+3ts+4{t}^{2}-3{t}^{2}s-\tfrac18{s}^{3
}-\tfrac52{t}^{3} \big) {x}^{3} \\&+\big( \tfrac14{s}^{2}+{\tfrac {1}{64}}{
s}^{4}-\tfrac18{s}^{3}+ts+{t}^{2}-\tfrac14t{s}^{2}-\tfrac12{t}^{2}{s}^{2}-\tfrac12
{t}^{4}-{t}^{3}s \big) {x}^{2}+\tfrac1{32}t ( -4s+{s}^{2}-8t+6
ts+6{t}^{2} ) ^{2}x,
\end{align*}
}%
dont le point $(P_{\infty,1}-P_{\infty,2})$  est d'ordre 7. On finit les calculs comme dans le cas de la première famille, pour aboutir à une équation quartique plane, à deux paramètres.
\end{proof}
\paragraph{Corps finis}

Une autre approche consiste à se placer sur les corps finis, ce qui permet d'obtenir des courbes dont la jacobienne possède des points d'ordre $l$. Voici un algorithme probabiliste, qui fournit, sur $\F_p$, des courbes à multiplication réelle par $\Q(\zeta_l^+)$.

\begin{algo}[H]\caption{Courbes à multiplication réelle par $\Q(\zeta_l^+)$ (II)}\label{algo:mult:reelle2}
\noindent\textbf{Entrée} : $(l,p)$ 
\begin{algorithmic}[1]
\Repeat
\State \textbf{Choisir} $\lambda_1,\lambda_2,\lambda_3\in\F_p$
\State $H\leftarrow y^2-x(x-1)(x-\lambda_1)(x-\lambda_2)(x-\lambda_3)$
\State $n\leftarrow\#\Jac(H)=:J$
\Until{$n\equiv 0\mod l$}
\State{$m\leftarrow\tfrac nl$}
\Repeat
\State \textbf{Choisir} $Q\in\Jac(H)\setminus\{O_J\}$.
\Until{$mQ\neq O_H$}
\State{$P\leftarrow mQ$}\Comment{$P$ est d'ordre $l$}
\State{\textbf{Calculer} $\phi(x,y)$ telle que $\div(\phi)=l\Div(P)$}
\State{$\Pi\leftarrow\Res_t(\phi(x,-y)t^l-\phi(x,y),t^2+1-zt)$}
\State{$X\leftarrow \Pi\mod\!_y\  H$}\Comment{$X$ est un polynôme en $x,z$.}
\end{algorithmic}
\noindent\textbf{Sortie} : $X$
\end{algo}

\begin{expl}Finissons par donner, comme dans le cas précédent, deux exemples.
\begin{itemize0}
\item[\qquad\bf$\mathbold{l=}$\,5.]
L'algorithme ci-dessus fournit, dans $\F_{31}$, la courbe hyperelliptique $H$ d'équation de \WS  $y^2=x(x-1)(x-2)(x+4)(x-5)$ dont la jacobienne possède 1120 éléments.  $P=[x(x-7), 20x]$, en représentation de Mumford, est d'ordre 5 et la fonction
$x^2(4y+3x^3+7x^2+3x$ a comme diviseur $5((0,0,1)+(7,16,1)-2(0,1,0))$ : elle conduit à la courbe $X$ dont on donne une équation hyperelliptique $y^2=(x+18)(x^5 + 7x^4 + x^3 + 25x^2 + 7x + 12)$.
Le polynôme caractéristique de la trace du Frobenius, est
$r^2+19r+89$, de discriminant 5 et définissant bien le corps de nombre $\Q(\zeta_5^+)=\Q(\sqrt5)\hookrightarrow\Q(\pi_H)$. 
\item[\qquad\bf$\mathbold{l=}$\,7.]
Pour $p=41$, l'algorithme~\ref{algo:mult:reelle2} propose $y^2=x(x-1)(x-3)(x+7)(x-7)$ pour l'équation de $H$, dont la jacobienne possède $7\cdot240$ points. Parmi eux, en représentation de Mumford, $[(x+5)^2,-2x+16]$, est d'ordre 7. La fonction \[(x^4+22x^3+22x^2+17x+24)y+10x^7+20x^6+35x^4+14x^3+40x^2+6x+36\] a pour diviseur $14((36,26,1)-(0,1,0))$. On calcule alors l'équation de $X$ sur $\F_{41}$, sous forme quartique,
{\footnotesize
\begin{equation*}
u^{3}v+u^{2}v^{2}+26uv^{3}+2v^{4}+24u^{3}+11u^{2
}v+21uv^{2}+2v^{3}+24u^{2}+37uv+3v^{2}+20u+11v+3.
\end{equation*}
}%
On vérifie à nouveau, par un calcul de fonction zêta, que le polynôme caractéristique de la trace du Frobenius est  $r^3+21r^2+126r+203$
qui définit le corps de nombres $\Q(\zeta_7^+)$.
\end{itemize0}
\end{expl}
\section{Multiplication réelle par \texorpdfstring{$\Q(\zeta_l^{(4)})$}{Q(zeta\_l(4))}.}
Dans toute cette section, nous considérons naturellement des entiers premiers $l$ congrus à 1 modulo 4. De plus, on se place ici plus volontiers sur le corps $k=\Q(i)$ ou éventuellement sur des corps de caractéristique $p$ avec $p\equiv1\mod 4$ de façon à encore disposer d'une racine carrée de $-1$. 
\subsection{Type (0,1,1)}\label{type:011}
On a, à nouveau, un zéro dans le type du revêtement : on cherche donc un revêtement $Z\rightarrow \Pp^1$, de degré 4, ramifié en deux points, d'ordre 4 aussi. La formule d'Hurwitz donne
\[g(Z)=\frac12\big(2+4(2g(\Pp^1)-2)+(4-1)+(4-1)\big)=0.\]
Une équation de  $Z$, de genre 0, peut être  $y^4=x$ avec le revêtement de $\Pp^1$, $(x,y)\mapsto x$, ramifié en $0$ et $\infty$, dont le groupe de Galois est $\Zz4$, engendré par $(x,y)\mapsto(x,iy)$.

On cherche alors un revêtement $Y\rightarrow Z$ de groupe $\Zz l$ tel que le revêtement $Y\rightarrow\Pp^1$ soit de type (0,1,1). Comme d'habitude on le cherche sous la forme d'une extension de Kummer $t^l=\phi(x,y)=\psi(y)$. Comme on travaille au-dessus de $\Pp^1$, on ne peut pas espérer de ramification annulée par un diviseur dont l'ordre en chacun des points de son support est multiple de $l$, car sur $\Pp^1$, seules les puissances $l$-ièmes ont de tels diviseurs \og multiples \fg{} de $l$.

Ainsi, la ramification est donnée exactement par les pôles et les zéros de $\psi(y)$. Comme on ne veut qu'un seul point de ramification au-dessus de $\Pp^1$, les pôles et les zéros de $\psi$ doivent avoir les mêmes puissances 4\ieme. À homographie près, on a donc essentiellement un seul revêtement convenable, fourni par la proposition suivante.

\begin{prop}
Soit $\psi(y)=\tfrac{(y-1)(y-i)}{(y+1)(y+i)}$ et $a\in(\Zz l)^*$ d'ordre $4$. On définit la courbe\vspace*{-3pt}
\[%
(Y)\ \Bigg\{\begin{array}{l@{\,=\,}l}
y^4&x\\
t^l&\psi(y)\psi(iy)^a.
\end{array}
\]
Alors, le revêtement $Y\rightarrow\Pp^1$ donné par l'application $(x,y,t)\mapsto x$ est de type (0,1,1). De plus, en considérant $k$ tel que $kl-a^2=1$, les automorphismes\vspace*{-5pt}
\begin{align*}
\alpha:(x,y,t)&\mapsto(x,iy,t^{-a}\psi(iy)^k)\\
\sigma:(x,y,t)&\mapsto(x,y,\zeta_l y)\vspace*{-5pt}
\end{align*}
engendrent le groupe métacyclique $G_{l,4}$, groupe de Galois de ce revêtement.
\end{prop}
\begin{proof}
On vérifie que le revêtement est bien celui attendu. Le revêtement intermédiaire $(x,y,t)\mapsto(x,y)$ est ramifié en les zéros et pôles de $\psi$, à chaque fois d'ordre $l$. On a 4 points de ramification $(1,1)$, $(1,-1)$, $(1,i)$ et $(1,-i)$. Ces quatre points ont même image, 1, par $(x,y)\mapsto x$.  Enfin, on a bien une ramification d'ordre 4 au-dessus des points $x=0$ et $x=\infty$.
On calcule ensuite\vspace*{-7pt}
\[
\psi(iy)\psi(i\cdot iy)^\alpha)=\psi(-y)^\alpha\psi(iy)=\psi(y)^{-\alpha}\psi(iy)^{kl-a^2}
=(\psi(y)\psi(iy)\big)^{-a}\psi(iy)^{kl}=\big(t^{-a}\psi(iy)^k\big)^l,\vspace*{-7pt}
\]
ce qui montre l'existence du morphisme $\alpha$, d'ordre 4 (on vérifie que $\alpha^2:(x,y,t)\mapsto(x,-y,\tfrac1t)$). Quant à $\sigma$, il découle comme d'habitude de la forme des extensions de type Kummer. Il reste à montrer que ces deux morphismes engendrent bien $G_{l,4}$. Pour cela, on effectue le calcul suivant.\vspace*{-5pt}
\[
\alpha\sigma\alpha^{-1}(x,y,t)=\alpha\sigma(x,-iy,t^a\psi(iy)^{-k})=\alpha(x,-iy,\zeta_l^at^a\psi(iy)^{-k})=(x,y,\zeta_l^{a}t)=\sigma^{a}(x,y,t).\qedhere
\]
\end{proof} 

Il reste à déterminer une équation de $X=Y/\langle\alpha\rangle$ de façon analogue  à ce que l'on a vu pour le type \hyperref[type:01111]{$(0,1,1,1,1)$}, quoiqu'un peu plus délicat. En effet, il s'agit de calculer le polynôme minimal d'un élément stabilisé par $\alpha$, par exemple\vspace*{-7pt}
\[z_t:=t+\alpha(t)+\alpha^2(t)+\alpha^3(t)=t+t^{-a}\psi(iy)^k+t^{-1}+t^a\psi(iy)^{-k}.\vspace*{-7pt}\]
On en déduit un algorithme partageant des techniques avec l'algorithme~\ref{algo:mult:reelle}, en se plaçant sur $\Q(i)$, \cad en introduisant un élément $i$ et la relation $i^2=-1$.

\vspace*{-5pt}
\begin{algo}[H]\caption{Courbe à multiplication réelle par $\Q(\zeta_l^{(4)})$ (I)}\label{algo:mult:reelle4}
\noindent\textbf{Entrée} : $l$
\begin{algorithmic}[1]
\State{\textbf{Factoriser} $u^2+1\mod l$ et choisir $a$, une racine. }
\State{$\psi(y)\leftarrow (y-1)(y-i)/(y+1)/(y+i)$}
\State{$\Pi(y,z)\leftarrow\Res_t\big(t^l-\psi(y)\psi(iy)^a,t+t^{-a}\psi(iy)^k+t^{-1}+t^a\psi(iy)^{-k}-z\big)$}\label{algo:mult:reelle4:ligne3}
\State{$X(x,z)\leftarrow \Pi(y,z)\mod\!_y\  y^4-x$}\label{algo:mult:reelle4:ligne4}
\end{algorithmic}
\noindent\textbf{Sortie} $X$.
\end{algo}
\vspace*{-8pt}

Notons l'abus de notation de la ligne~\ref{algo:mult:reelle4:ligne3} : on prend le résultant des numérateurs des quantités décrites, et comme pour chaque calcul de résultant, il est possible qu'il faille éliminer les facteurs non pertinents. Enfin, pour la ligne~\ref{algo:mult:reelle4:ligne4}, $\Pi$ est par construction un polynôme invariant par $\alpha$ et donc n'a  que des puissances 4\ieme de $y$, ce qui justifie le fait que $X$ soit un polynôme en $x$ et $z$ uniquement.
\begin{expl} On choisit $l=13$ et $a=5$ d'ordre 4 dans $(\Zz{13})^*$ et on travaille sur $\Q(i)$.  On trouve comme équation de $Y$\vspace*{-10pt}
\[
(Y)\ \Bigg\{\begin{array}{l@{\,=\,}l}
y^4&x\\
t^{13}&\frac{(y+i)^4(y-1)^6}{(y-i)^4(y+1)^6}
\end{array}
\vspace*{-4pt}\]
puis une équation de $X$, de genre 3,

{\footnotesize
\relpenalty=0000
\binoppenalty=0000
\noindent
$
(x\!-\!1)^6z^{13}\!-\!26(x\!-\!1)^6z^{11}\!+\!221(x\!-\!1)^6z^9\!-\!104(x^2\!+\!14x\!+\!1)(x\!-\!1)^4z^8\!-\!26(31x^2\!-\!126x\!+\!31)(x\!-\!1)^4z^7\!+\!884(x^2\!+\!14x\!+\!1)(x\!-\!1)^4z^6\!+\!26(31x^2\!+\!66x\!+\!31)(x\!-\!1)^4z^5\!+\!52(\!-\!37x^4\!-\!380x^3\!+\!256ix^3\!+\!1858x^2\!-\!380x\!-\!256ix\!-\!37)(x\!-\!1)^2z^4\!+\!52(25x^4\!+\!252x^3\!+\!4566x^2\!+\!252x\!+\!25)(x\!-\!1)^2z^3\!-\!416(x^4\!-\!44x^3\!+\!32ix^3\!-\!810x^2\!-\!44x\!-\!32ix\!+\!1)(x\!-\!1)^2z^2\!+\!13(5x^4\!-\!1684x^3\!+\!15646x^2\!-\!1684x\!+\!5)(x\!-\!1)^2z\!-\!4\!+\!4824x\!-\!141312ix^4\!+\!5120ix^5\!-\!5120ix\!+\!141312ix^2\!+\!4824x^5\!+\!42180x^4\!-\!4x^6\!-\!356144x^3\!+\!42180x^2.
$
}%

\noindent
Cette dernière, non hyperelliptique possède une équation quartique plus simple,\vspace*{-3pt}
\[u^3+2u^2v+2u^2-uv^3-2uv^2-2uv-2u+v^4+v^2-v+2.\vspace*{-3pt}\]
En réduisant par exemple cette équation dans $\F_{53}$, on trouve une jacobienne d'une courbe non supersingulière, dont le polynôme caractéristique de la trace du Frobenius,\vspace*{-4pt}
\[r^3-6r^2-40r-8\vspace*{-4pt}\]
définit le corps de nombres $\Q(\zeta_{13}^{(4)})$. On a choisi expressément $p=53$ de façon à ce que $\F_p$ possède les racines 13\ieme de l'unité. La table~\ref{table:expl}, rassemble quelques exemples de polynômes caractéristiques de la trace du Frobenius pour des réductions de $X$ dans divers $\F_p$.
\begin{table}[ht]
\centering
\begin{tabular}{|c|c||c|}
\hline\multirow{2}{*}{$\F_p$}&Ordre de&
{Polynôme caractéristique}\\[-3pt]
&$l\in(\Zz{13})^*$&de la trace du Frobenius\\
\hline
$53$&$1$&$r^3-6r^2-40r-8$\\
$181$&$2$&{$r^3+39r^2+416r+689$}\\
$73$&$4$&$r^3-12r^2-43r+3269
$\\
$29$&$3$&$r^3-87r-80$, ($t^6-80t^3+29^3$)\\
$17$&$6$&$r^3-51r-52$, ($t^6 - 52t^3 + 17^3$)\\
$37$&$12$&$r^3-111r-306$, ($t^6 - 306t^3 + 37^3$)\\\hline
\end{tabular}
\caption{Polynômes caractéristiques pour des réductions de courbes $X$ à multiplication réelle par $\Q(\zeta_{13}^{(4)})$ sur des corps finis (I).}
\label{table:expl}
\end{table}

Pour les trois premières lignes, on a en fait $\zeta_{13}^{(4)}\in\F_p$ et les trois polynômes caractéristiques définissent le corps de nombres  $\Q(\zeta_{13}^{(4)})$. Par contre, on voit pour les trois dernières lignes que la courbe devient supersingulière. On vérifie, en calculant le résultant en $x$ des polynômes caractéristiques $P(x)$ et $z-x^k$, que sur une extension de degré $k=3$, on obtient en fait des jacobiennes isogènes à un produit de 3 courbes elliptiques identiques. On peut aussi le vérifier sur le polynôme caractéristique de la trace du Frobenius. Par exemple\vspace*{-4pt}
\begin{align*}&\Res_x(x^6-80x^3+29^3,z-x^3)=(z^2-80z+29^3)^3,\\[-3pt]
&\Res_t(t^3-87t-80,\Res_x(z-x^3,tx-x^2-29)=-(z^2-80z+29^3)^3.\\[-20pt]
\end{align*}
Dès lors, $\End_\Q(\Jac(X))$ contient au moins l'algèbre $\Mc_3(\Q)$ qui, grâce aux matrices compagnons, contient tous les corps de nombres de degré 3, et en particulier $\Q(\zeta_{13}^{(4)})$. On remarque enfin que dans ces trois cas, $\zeta_{13}^{(4)}\in\F_{l^3}$, extension minimale, de degré 3.
\end{expl}

\subsection{Type (1,1,2,2)}\label{type:1122}
On cherche ici un recouvrement $H\rightarrow\Pp^1$ de degré 4, ramifié en 4 points, deux d'ordre 2 et deux d'ordre 4. On peut calculer le genre  de $H$ par la formule d'Hurwitz, même s'il y a plusieurs possibilités. En effet, pour la ramification d'ordre 2, on peut avoir 2 points d'ordre 2 dans la fibre ou 1 point d'ordre 2 et deux autres points. Pour des raisons de parité, cela laisse deux possibilités
\begin{align*}
g(H)&=\tfrac12(2+4(2g(\Pp^1)-2)+2(4-1)+4(2-1))=2\\[-5pt]
\intertext{ou}\\[-30pt]
g(H)&=\tfrac12(2+4(2g(\Pp^1)-2)+2(4-1)+2(2-1))=1.
\end{align*}
La deuxième possibilité ne nous intéresse pas ici car on veut ensuite un revêtement non ramifié d'ordre $l$, qui serait alors encore une courbe elliptique, ce qui n'est pas possible pour $Y$. 
\begin{prop}
Soit $\tau\in\Q$ un paramètre, différent de $2$ et $-2$. et soit $H$ la  courbe hyperelliptique de genre $2$, d'équation $y^2=x(x^4+\tau x^2+1)$. Alors, l'application $(x,y)\mapsto x^2$ est un revêtement de $\Pp^1$, de degré $4$ et de type $(1,1,2,2)$. Son groupe de Galois est $\Zz4$, engendré par l'automorphisme $(x,y)\mapsto(-x,iy)$.
\end{prop}
\begin{proof}
Ce revêtement  est ramifié en $0$ et $\infty$, de degré 4. De plus, notons $\omega$ une racine de $x^4+\tau x^2+1$. Alors, les quatre racines distinctes ($\tau\neq\pm2$)  sont $\pm\omega$ et $\pm\tfrac1\omega$. On a une ramification d'ordre 2 en ces points de \WS, et l'application $(x,y)\mapsto x^2$ identifie ces points par paires : le revêtement est  ramifié en deux points supplémentaires, $\omega^2$ et $\tfrac1{\omega^2}$, d'ordre 2, en plus des points $0$ et $\infty$ d'ordre 4.
\end{proof}

Ici, le groupe d'automorphismes de $H$ est $\Zz2\times\Zz4$, avec $(x,y)\mapsto(\tfrac1x,\tfrac y{x^3})$ en plus. Néanmoins, l'automorphisme $(x,y)\mapsto(-\tfrac1x,i\tfrac y{x^3})$ est d'ordre 4, mais ne convient pas ici car il ne stabilise pas les fibres.

La suite peut se traiter de manière similaire au type~\hyperref[type:111111]{$(1,1,1,1,1,1)$}, en construisant un revêtement non ramifié d'ordre $l$ grâce à un point de $\Jac(H)$ d'ordre $l$.
En fait, la situation est ici bien simplifiée par l'existence d'un morphisme non constant de $H$ vers une courbe elliptique $E$, quotient de $H$ par le groupe d'ordre 2 engendré par $(x,y)\mapsto(\tfrac1x,\tfrac{y}{x^3})$. La jacobienne de $H$ est isogène à un  produit de deux courbes elliptiques isomorphes à $E$. C'est en fait un cas particulier, pour $l=2$ des courbes qui apparaissent dans~\cite{top} et que l'on a croisées en~\ref{type:0011}. On vérifie de manière immédiate la proposition suivante.

\begin{prop}\label{decomp:H1:E2}
Il existe un morphisme entre la courbe $H$ et la courbe elliptique $E$ d'équation $Y^2=(X+2)(X^2+\tau-2)$. Il est donné par les équations (affines)\vspace*{-5pt}
\begin{equation}
\begin{array}{l@{\,=\,}l}
X&x+\frac1x\\
Y&\frac{y(x+1)}{x^2}.
\end{array}\label{type1122:mor}
\end{equation}
\end{prop}

Ce morphisme repose essentiellement sur l'identité $x(X+2)=(x+1)^2$, qui relie $x$ et $X$ à un carré près. En changeant les signes, on a aussi $x(X-2)=(x-1)^2$, ce qui donne la deuxième courbe isomorphe $E'$ d'équation $Y'^2=(X-2)(X^2+\tau-2)$, avec $Y'=\tfrac{y(x-1)}{x^2}$ et l'isomorphisme $(X,Y)\mapsto(-X,iY)$.
On se sert de la courbe elliptique $E$ pour construire un revêtement non ramifié de $H$ de degré $l$.

\begin{prop}\label{mult:reelle:4}
Soit $P$ un point de $E$ d'ordre $l$ et soit $a\in(\Zz l)^*$ d'ordre $4$. On considère une fonction $\phi(X,Y)$ de diviseur $l((P)-(O_E))$ et $\psi(x,y)$ la fonction correspondante sur $H$. On définit alors la courbe\vspace*{-5pt} 
\[%
(Y)\ \Bigg\{\begin{array}{l@{\,=\,}l}
y^2&x(x^4+\tau x^2+1)\\
t^l&\psi(x,y)\psi(-x,iy)^a.
\end{array}\vspace*{-5pt} 
\]
Alors, le revêtement $Y\rightarrow\Pp^1$, $(x,y,t)\mapsto x^2$ est de type $(1,1,2,2)$, de groupe de Galois $G_{l,4}$, engendré par les automorphismes\vspace*{-15pt} 
\begin{align*}
\alpha:(x,y,t)&\mapsto(-x,iy,t^{-a}\psi(-x,iy)^k\Phi(x)^a)\\
\sigma:(x,y,t)&\mapsto(x,y,\zeta_l t)\\[-23pt]
\end{align*}
où $k$ est tel que $kl-a^2=1$, et $\Phi(x)$ est une fonction que l'on précisera.
 
\end{prop}

\begin{proof}
Comme $lP=O_E=lO_E$ le diviseur $l((P)-(O_E))$ est principal, ce qui justifie l'existence de $\phi$ puis $\psi$ par composition avec le morphisme~(\ref{type1122:mor}). Comme $\psi(x,y)\psi(x,-y)$ est stable par l'involution hyperelliptique, c'est une fonction de $x$ seulement, dont le diviseur, vu sur $\Pp^1$, ne possède que des points d'ordre $l$. C'est donc nécessairement le diviseur d'une puissance $l$-ième, que l'on note $\Phi(x)^l$.

Le revêtement $Y\rightarrow H$ est non ramifié, car le diviseur de $\psi(x,y)\psi(-x,iy)^a$ a par construction un support \og multiple \fg de $l$. C'est un revêtement habituel de type Kummer, de degré $l$ et de groupe de Galois $\Zz l$. Ainsi, la composition par le revêtement $H\rightarrow\Pp^1$ est bien du type annoncé $(1,1,2,2)$.
Pour justifier que ce revêtement composé est galoisien de groupe de Galois $G_{l,4}$ on montre que l'automorphisme de $H$ de degré 4 se \og remonte \fg{} sur $Y$,\vspace*{-8pt} 
\begin{align*}
\psi(-x,iy)\psi(x,-y)^a=\psi(-x,iy)\psi(x,y)^{-a}\Phi(x)^{al}
&=\big(\psi(x,y)\psi(-x,iy)^a\big)^{-a}\psi(-x,iy)^{kl}\Phi(x)^{al}\\
&=\big(t^{-a}\psi(-x,iy)^k\Phi(x)^a\big)^l\\[-23pt]
\end{align*}
ce qui justifie l'existence de l'automorphisme $\alpha$. On calcule enfin\vspace*{-6pt} 
\begin{align*}
\alpha\sigma\alpha^{-1}(x,y,t)&=\alpha\sigma(-x,-iy,t^a\psi(-x,iy)^{-k}\Phi(-x))\\&=\alpha(-x,-iy,\zeta_l^at^a\psi(-x,iy)^{-k}\Phi(-x))=(x,y,\zeta_l^{a}t)=\sigma^{a}(x,y,t).\qedhere
\end{align*}
\end{proof}

Le calcul d'une équation de la courbe $X/\langle\alpha\rangle$ se fait comme précédemment en calculant le polynôme minimal d'un élément invariant par $\langle\alpha\rangle$, par exemple\vspace*{-8pt}
\[%
z_t:=\!\sum_{i=0}^3\alpha^i(t)=t+t^{-a}\psi(-x,iy)^k\Phi(x)^a+\frac{\Phi(x)\Phi(-x)^a}t+t^a\psi(-x,iy)^{-k}\Phi(-x).\vspace*{-5pt}
\] 
Pour calculer une équation de  $X$, on prend $\Res_t(z-z_t,t^l-\psi(x,y)\psi(-x,iy)^a)$, qui est invariant par $(x,y)\mapsto(-x,iy)$ et donc par $y\mapsto-y$ ; on élimine $y$ en prenant le reste modulo $y^2-x(x^4+\tau x^2+1)$. Le résultat est invariant par $x\mapsto -x$ et une équation de $X$ s'obtient en prenant le reste par $w-x^2$.
Le tout aboutit à la proposition suivante.

\begin{prop}
Les courbes construites ci-dessus sont sont une famille à un paramètre, $\tau$, dont la jacobienne est à multiplication réelle par $\Q(\zeta_{l}^{(4)})$.
\end{prop}

\begin{proof}
Il suffit de trouver une fonction $\psi$ qui provient un point de $l$\ti torsion de $E$, donné par une racine du polynôme de $l$\ti division. On peut aussi utiliser $X_1(l)$ qui paramétrise les courbes elliptiques possédant un point de $l$\ti{}torsion.
\end{proof}


Si comme dans le cas~\ref{type:111111} tous les calculs exposés ci-dessus sont explicites, ils semblent exiger une puissance trop importante à l'heure actuelle. On peut toutefois en tirer un algorithme qui fournit des courbes à multiplications réelles par $\Q(\zeta_l^{(4)})$ sur un corps fini $\F_p$, aléatoirement parmi cette famille à 1 paramètre.
Pour s'assurer que les courbes sont définies sur $\F_p$ on choisit $p\equiv1\mod 4$ afin de disposer d'une racine carrée de $-1$ dans $\F_p$. 
Enfin, une difficulté peut venir de $\Phi$ : définie à partir d'un diviseur, on peut avoir $\phi(x,y)\phi(x,-y)=k\Phi(x)^l$. Mais alors, il suffit de remplacer $\phi$ par $\tilde\phi=k^{\frac{l-1}2}\phi$. Ce détail réglé, on en déduit l'algorithme suivant.

\begin{algo}\caption{Courbes à multiplication réelle par $\Q(\zeta_l^{(4)})$ (II)}\label{algo:mult:reelle4b}
\noindent\textbf{Entrée} : $l,p$
\begin{algorithmic}[1]
\State{\textbf{Factoriser} $u^2+1\mod l$ et choisir $a$, une racine. }
\Repeat
\State{\textbf{Choisir $\tau\in\F_p$}}
\State{$E\leftarrow Y^2-(X+2)(X^2+\tau-2)$}
\Until{$\#E\equiv0\mod l$}
\State{\textbf{Choisir} $P\in E[l]\setminus\{O_E\}$}
\State{\textbf{Calculer} $\phi(x,y)$ vérifiant $\div(\phi)=l((P)-(O_E))$ de coefficient dominant adéquat}
\State{$\psi(x,y)\leftarrow\phi(x+\tfrac1x,\tfrac{y(x+1)}{x^2})$ et $\Phi(x)\leftarrow(x+1/x-x_P)$}
\State{$\Pi\leftarrow\Res_t(z-z_t,t^l-\psi(x,y)\psi(-x,iy)^a)$}
\State{$X(w,z)\leftarrow \big(\Pi \mod\!_{y}\ y^2-x(x^4+\tau x^2+1)\big) \mod\!_x\ w-x^2$}
\end{algorithmic}
\noindent\textbf{Sortie} $X$.
\end{algo}
\begin{expl}On donne, dans la table~\ref{table:eq:quart}, des équations quartiques à multiplication réelle par $\Q(\zeta_{13}^{(4)})$ pour différentes valeurs de $p$. On fournit dans la table~\ref{table:expl2} les polynômes caractéristiques du Frobenius, permettant de vérifier que les courbes obtenues sont à multiplication complexe par une extension quadratique de $\Q(\zeta_{13}^{(4)})$.
\begin{table}[ht]
\centering
\begin{tabular}{|c|c||l|}\hline
$\F_p$&$\tau$&\hfill Équation quartique\hfill\phantom{.} \\\hline
\multirow{2}{*}{$53$}&\multirow{2}{*}{$7$}&{\footnotesize$u^{4}+42u^{3}v+38u^{2}v^{2}+5uv^{3}+7v^{4}+37u
^{3}+26u^{2}v+37uv^{2}$}\\[-4pt]
&&{\footnotesize$\phantom{u^4}+24v^{3}+13u^{2}+41uv+37v^
{2}+21u+40v+26$}\\\hline
\multirow{2}{*}{$181$}&\multirow{2}{*}{$24$}&{\footnotesize$u^{4}+28u^{3}v+5u^{2}v^{2}+152uv^{3}+99v^{4}+151
u^{3}+126u^{2}v+22uv^{2}$}\\[-4pt]
&&{\footnotesize$\phantom{u^4}+34v^{3}+92u^{2}+158uv+157
v^{2}+74u+140v+61$}\\\hline
\multirow{2}{*}{$73$}&\multirow{2}{*}{$32$}&{\footnotesize$u^{4}+71u^{3}v+3u^{2}v^{2}+5uv^{3}+70v^{4}+70u
^{3}+13u^{2}v+49uv^{2}$}\\[-4pt]
&&{\footnotesize$\phantom{u^4}+56v^{3}+20u^{2}+32uv+18v^
{2}+39u+59v+68$}\\\hline
\multirow{2}{*}{$29$}&\multirow{2}{*}{$19$}&{\footnotesize$u^{4}+4u^{3}v+26u^{2}v^{2}+3uv^{3}+7v^{4}+12u^
{3}+24u^{2}v+28uv^{2}$}\\[-4pt]
&&{\footnotesize$\phantom{u^4}+26v^{3}+5u^{2}+6uv+3v^{2}+
24u+22v+17$}\\\hline
\multirow{2}{*}{$101$}&\multirow{2}{*}{$15$}&{\footnotesize$u^{4}+37u^{3}v+85u^{2}v^{2}+77uv^{3}+97v^{4}+65{
u}^{3}+10u^{2}v+57uv^{2}$}\\[-4pt]
&&{\footnotesize$\phantom{u^4}+13v^{3}+100u^{2}+67uv+86{
v}^{2}+75u+12v+54$}\\\hline
\multirow{2}{*}{$41$}&\multirow{2}{*}{$3$}&{\footnotesize$u^{4}+19u^{3}v+4u^{2}v^{2}+2uv^{3}+5v^{4}+28u^
{3}+11u^{2}v$}\\[-4pt]
&&{\footnotesize$\phantom{u^4}+33v^{3}+17u^{2}+30uv+29v^{2}+32v+29$}\\\hline
\end{tabular}
\caption{Quartiques à multiplication réelle par $\Q(\zeta_{13}^{(4)})$ sur des corps finis.}\label{table:eq:quart}
\end{table}

Dans les trois premières lignes de la table~\ref{table:expl2}, $\zeta_{13}^{(4)}$ est dans le corps de base $\F_p$ : le corps de rupture du polynôme caractéristique du Frobenius est bien une extension quadratique de $\Q(\zeta_{13}^{(4)})$. Dans les trois dernières lignes, on obtient une courbe supersingulière dont la jacobienne, sur une extension de degré 3, est isogène au cube d'une courbe elliptique.
\begin{table}[ht]
\centering
\begin{tabular}{|c|c|c||c|}
\hline\multirow{2}{*}{$\F_p$}&Ordre de&\multirow{2}{*}{$\tau$}&
{Polynôme caractéristique }\\[-3pt]
&$l\in(\Zz{13})^*$&&de la trace du Frobenius\\
\hline
$53$&$1$&$7$&{ $r^3+20r^2+77r-46$}\\
$181$&$2$&$24$&{ $r^3+26r^2+13r-1625$}\\
$73$&$4$&$32$&{$r^3+40r^2+529r+2315$}\\
$29$&$3$&$19$&$r^3-87r+63$, ($t^6 + 63t^3 + 29^3$)\\
$101$&$6$&$17$&$r^3-303r+286$, ($t^6 + 286t^3 + 101^3$)\\
$41$&$12$&$12$&$r^3-123r-30$, ($t^6 - 30t^3 + 41^3$)\\\hline
\end{tabular}
\caption{Polynômes caractéristiques pour des réductions de courbes $X$ à multiplication réelle par $\Q(\zeta_{13}^{(4)})$ sur des corps finis (II).}
\label{table:expl2}
\end{table}
\end{expl}

\vspace*{-15pt}
\section{Multiplication réelle par \texorpdfstring{$\Q(\zeta_l^{(k)})$ pour $k=6$, $8$ et $10$.}{Q(zeta\_l(k)) pour k=6, 7 et 10.}}

On peut donner un traitement similaire à la section~\ref{type:1122}, notamment dans la manière de construire la courbe $Y$ comme une extension de Kummer d'un produit de fonctions, où l'on fait agir l'automorphisme de la première courbe, afin d'obtenir une extension \emph{globalement} galoisienne. Ainsi, dans le cas~\ref{type:1122} on a dû considérer la fonction $\psi(x,y)\psi(-x,iy)^a$ et non pas $\psi(x,y)$. Ici, on peut faire de même, avec respectivement 3, 6, 4 et 5 facteurs pour les types  \hyperref[type:2233]{$(2,2,3,3)$}, \hyperref[type:112]{$(1,1,2)$}, \hyperref[type:114]{$(1,1,4)$} et \hyperref[type:125]{$(1,2,5)$}.

On propose à la place d'utiliser dans la manière du possible, les propriétés, notamment liées à la multiplication complexe, de la courbe sur laquelle on considère l'extension de Kummer.  

Pour les deux sections suivantes, l'indice est 6 et on considère des entiers premiers~$l$ congrus à 1 modulo 6. On se place sur $\Q(j)$ avec $j^3=1$, ou sur des corps finis de caractéristique $p\equiv 1\mod 3$ de manière à disposer d'une racine 3\ieme de l'unité.

\subsection{Type (2,2,3,3)}\label{type:2233}

On commence encore une fois par la formule d'Hurwitz pour trouver un revêtement $H\rightarrow\Pp^1$ de degré 6, ramifié en 4 points, de type (2,2,3,3),
\[2g(H)-2=6(2g(\Pp^1)-2+r(3-1)+s,\]
où $r\in\{1,2,3,4\}$ est le nombre de points ramifiés dans les fibres d'ordre 3 et $s\in\{2,4,6\}$ est le nombre de points ramifiés dans les fibres d'ordre 2, nécessairement pair, d'après cette même formule d'Hurwitz. Cela impose donc $g(H)\in\{0,1,2\}$ et comme on veut ensuite un revêtement non ramifié d'ordre $l$, on a nécessairement $r=4$ et $s=6$ de façon à obtenir $g(H)=2$.

\begin{prop}
Soit $\tau\not\in\{-2,2\}$ un paramètre et soit $H$ la courbe hyperelliptique de genre $2$, définie par l'équation
\[y^2=x^6+\tau x^3+1.\]
Alors, l'application $H\rightarrow\Pp^1$, $(x,y)\mapsto x^3$ est un revêtement de type $(2,2,3,3)$, de degré $6$, tout comme l'automorphisme
\[(x,y)\mapsto(jx,-y),\]
 qui engendre le groupe de Galois, $\Zz6$, de ce revêtement. 
\end{prop}

\begin{proof}
Soient $\omega$ et $\omega'$ les racines non nulles et distinctes de $z^2+\tau z+1$. Le revêtement $(x,y)\mapsto x^3$ est ramifié en $0, \infty, \omega$ et $\omega'$. La fibre au-dessus de $0$ est constituée des deux points $(0,\pm1)$ tout comme celle au-dessus de $\infty$ constituée des deux points à l'infini.
Les fibres au-dessus de $\omega$ et $\omega'$ sont quant à elles constituées de 3 points de \WS chacune.
\end{proof}

De façon similaire à la construction associée au type \hyperref[type:1122]{$(1,1,2,2)$} on peut utiliser le fait que $\Jac(H)$ est le produit de deux courbes elliptiques $Y^2=(X\pm2)(X^3-3X+\tau)$ pour construire le revêtement de degré $l$, non ramifié.

\begin{prop}\label{mult:reelle:6}
Soit $P$ un point de $E$ d'ordre $l$ et soit $a\in(\Zz l) ^*$ d'ordre $3$. On considère une fonction $\phi(X,Y)$ de diviseur $l((P)-(O_E))$ et $\psi(x,y)$ la fonction correspondante sur $H$. On définit alors la courbe 
\[
(Y)\ \Bigg\{\begin{array}{l@{\,=\,}l}
y^2&(x^6+\tau x^3+1)\\
t^l&\psi(x,y)\psi(jx,y)^a\psi(j^2x,y)^{a^2}.
\end{array}\]
Alors, le revêtement $Y\rightarrow\Pp^1$, $(x,y,t)\mapsto x^3$ est de type $(2,2,3,3)$, de groupe de Galois $G_{l,6}$, engendré par les automorphismes
\begin{align*}
\alpha:(x,y,t)&\mapsto\big(jx,-y,t^{-a^2}\psi(jx,y)^k\psi(j^2x,y)^{ak}\Phi(jx)\Phi(j^2x)^a\Phi(x)^{a^2}\big)\\
\sigma:(x,y,t)&\mapsto(x,y,\zeta_l t)
\end{align*}
où $k$ est tel que $a^3=1+kl$, et $\Phi(x)$ est une fonction que l'on précisera.
\end{prop}

L'équation de $Y=X/\langle\alpha\rangle$, se calcule, de façon usuelle un polynôme invariant par $(x,y)\!\mapsto\!(jx,\!-y)$ :\vspace*{-8pt}
\[\Pi(x,y,z):=\Res_t\big(z-\sum\limits_{r=0}^5\alpha^r(t),z^l-\prod\limits_{r=0}^2\psi(j^rx,y)^{a^r}\big)\vspace*{-8pt}\], qui est donc un polynôme en $y^2$ et $x^3$ et l'équation de $X$ est donnée par $X(w,z)=\Pi(x,y,z)\mod\!_{x,y}\ [y^2-(x^6+\tau x^3+1),x^3-w]$.
Ceci montre, comme dans le cas~\ref{type:1122}, la proposition suivante. 

\begin{prop}
Les courbes construites ci-dessus sont une famille à un paramètre, $\tau$, dont la jacobienne est à multiplication réelle par $\Q(\zeta_{l}^{(6)})$. 
\end{prop}

\paragraph{Utilisation des sous-groupes stables} Toutefois, comme on l'a évoqué dans l'introduction de cette section, on peut \og simplifier \fg la courbe $Y$ en cherchant un point sur la jacobienne de $y^2=x^6+\tau x^3+1$ dont le groupe engendré est stable par $\alpha:(x,y)\mapsto(jx,-y)$. Sur les corps finis, on l'utilise dans l'algorithme suivant.

\begin{algo}[H]\caption{Courbes à multiplication réelle par $\Q(\zeta_l^{(6)})$}\label{algo:mult:reelle6}
\noindent\textbf{Entrée} : $l,p$
\begin{algorithmic}[1]
\Repeat
\State{\textbf{Choisir $\tau\in\F_p$}}
\State{$J\leftarrow\Jac(x^6+\tau x^3+1)$}
\Until{$\exists P\in J$ tel que $\langle P\rangle$ soit stable par $\alpha$}\label{alg2:5:l:4}
\State{\textbf{Calculer}  $a$ tel que $\alpha(P)=aP$} 
\State{\textbf{Calculer} $\phi(x,y)$ vérifiant $\div(\phi)=l\Div(P)$}
\State{\textbf{Calculer} $\Phi(x,y)$ tel que $\div(\Phi)=\alpha(P)-aP$}
\State{\textbf{Définir} $\alpha':(x,y,t)\mapsto(jx,-y,\Phi(x,y)t^a)$}
\State{$\Pi\leftarrow\Res_t\big(z-\sum\limits_{r=0}^5\alpha'^r(t),z^l-\phi(x,y)\big)$}
\State{$X(w,z)\leftarrow \Pi \mod\!_{x,y}\ [x^2-w,y^2-(x^6+\tau x^3+1)]$}
\end{algorithmic}
\noindent\textbf{Sortie} $X$.
\end{algo}

\begin{expl} On finit cette section en regardant deux exemples. Dans le premier, on considère le cas $l=13$. On applique l'algorithme précédent pour différentes caractéristiques $p$ et on regroupe les résultats dans la table~\ref{table:expl3} donnant les équations courbes hyperelliptiques et les polynômes caractéristiques. 

\begin{table}[ht]
\centering
\begin{tabular}{|c|c|c||c|}
\hline\multirow{2}{*}{$\F_p$}&Ordre de&\multirow{2}{*}{$\tau$}&Équation hyperelliptique de $X$\\[-2pt]
&$l\in(\Zz{13})^*$&&Polynôme caractéristique de la trace du Frobenius\\
\hline
\multirow{2}{*}{$79$}&\multirow{2}{*}{$1$}&\multirow{2}{*}{$15$}&{$y^2=x^6 + 34x^5 + 30x^4 + 65x^3 + 46x^2 + 30x + 9$ }\\[-3pt]
&&&$r^2-4r-9$\\
\multirow{2}{*}{$103$}&\multirow{2}{*}{$2$}&\multirow{2}{*}{$17$}&{$y^2=3x^6 + 94x^5 + 67x^4 + 83x^3 + 63x^2 + 93x + 23$}\\[-3pt]
&&&$r^2-13r+39
$\\
\multirow{2}{*}{$139$}&\multirow{2}{*}{$3$}&\multirow{2}{*}{$17$}&$y^2=2x^6 + 77x^5 + 61x^4 + 104x^3 + 118x^2 + 91x + 137$\\[-3pt]
&&&$r^2+27r+101$\\
\multirow{2}{*}{$127$}&\multirow{2}{*}{$6$}&\multirow{2}{*}{$12$}&$y^2=3x^6 + 20x^5 + 106x^4 + 125x^3 + 34x^2 + 73x + 63$\\[-3pt]
&&&$r^2+26r+156$\\
\multirow{2}{*}{$31$}&\multirow{2}{*}{$4$}&\multirow{2}{*}{$7$}&$y^2=3x^6 + 24x^5 + 7x^4 + 17x^3 + 14x^2 + 21x + 15$\\[-3pt]
&&&$r^2-75$, ($t^4 - 13t^2+31^2$)\\
\multirow{2}{*}{$37$}&\multirow{2}{*}{$12$}&\multirow{2}{*}{$14$}&$y^2=x^6 + 26x^5 + 31x^4 + 2x^3 + 9x^2 + 34x + 9$\\[-3pt]
&&&$r^2-100$, ($t^4 - 26t^2 + 37^2$)\\
\hline
\end{tabular}
\caption{Polynômes caractéristiques pour des réductions de courbes $X$ à multiplication réelle par $\Q(\zeta_{13}^{(6)})$ sur des corps finis.}
\label{table:expl3}
\end{table}
\noindent
Pour les ordres de $p\in(\Zz{13})^*$ divisant $6$, le polynôme caractéristique du Frobenius  définit un corps de nombres qui est une extension quadratique de $\Q(\zeta_{13}^{(6)})$. Dans les deux autres cas, ce sont des jacobiennes de courbes supersingulières, qui se scindent en un produit de deux courbes elliptiques sur une extension de degré 2.

Dans un second temps, on spécifie $\tau=0$, mais pour $l\equiv1\mod3$ général. Dans ce cas\footnote{On ne pouvait pas utiliser cela dans la section précédente du fait que l'on obtient nécessairement une fonction $\phi(x,y)$ en $x^2$ faisant tomber l'ordre de $(x,y)\mapsto(-jx,y)$ à 3.}, $H$, d'équation $y^2=x^6+1$,  possède une jacobienne non simple dont on peut prendre comme facteur elliptique $Y^2=X^3+1$, avec $Y=y$ et $X=x^2$. 

Cette courbe elliptique est à multiplication complexe par $\Q(j)$. Comme on a $l\equiv1\mod3$, l'idéal $(l)$ est décomposé dans $\Q(j)$. On écrit $a+bj$ un de ces facteurs et on voit $E$ comme $\C/(\Z+j\Z)$. Parmi les points de $l$\ti{}torsion de $E$, le sous-groupe engendré par $\tfrac{a+jb}{l}$ est stable par multiplication par $j$. Cela veut dire que le polynôme de $l$\ti{}division de $E$ se scinde sur $\Q(j)$, possédant deux facteurs de degré $\tfrac{l-1}2$, dont le produit est un polynôme à coefficients dans $\Q$ de degré $l-1$. On note $P_l=(x_0,y_0)$ un point de $l$\ti{}torsion, dont $x_0$ est racine de ce polynôme. Les coordonnées de $P_l$ sont donc dans un corps de nombres degré $2l-2$. On considère ensuite, comme précédemment une fonction $\phi(X,Y)$ de diviseur $l((P_l)-(O_E))$, puis la fonction $\psi(x,y)=\phi(x^2,y)$ sur $H$. Cette fonction nous permet, comme expliqué précédemment de considérer la courbe $Y$ d'équation
\[
(Y)\ \Bigg\{\begin{array}{l@{\,=\,}l}
y^2&(x^6+1)\\
t^l&\psi(x,y).
\end{array}\]
Pour $l=13$, les calculs aboutissent et donnent les deux propositions suivantes. Notons que la courbe hyperelliptique de genre 2 que l'on obtient possède des invariants absolus définis sur $\Q(j)$.
\begin{prop}
Soit $x_0$ une racine du polynôme $x^6 + (12j + 8)x^3 + \tfrac1{13}(48j + 64)$ et $y_0$ tel que $y_0^2=x_0^3+1$. Alors, le point $(x_0,y_0)$ de la courbe elliptique $y^2=x^3+1$ est d'ordre 13, et le sous-groupe qu'il engendre contient $(jx_0,y_0)$.
\end{prop}
\begin{prop}
La courbe hyperelliptique d'équation

{\relpenalty=0000
\binoppenalty=0000
\footnotesize
\noindent
$Y^2=\big(X+ ( {\frac {1}{468}} ( -110j-115 ) x_0^{5
}+{\frac {1}{234}} ( -451j+289 ) x_0^{2}
 ) {y_0}\alpha+ ( {\frac {1}{156}} ( 67j+47
 ) x_0^{5}+\frac1{39} ( 73j-85 ) x_0^{2}
 ) {y_0}\big)
 \big(X+ ( {\frac {1}{936}} ( 653j+337
 ) x_0^{5}+{\frac {1}{234}} ( 314j-1331 ) 
x_0^{2} ) {y_0}\alpha+ ( \frac1{24} ( 47j+19
 ) x_0^{5}+\frac16 ( 5j-110 ) x_0^{2}
 ) {y_0}\big)\big({X}^{2}+ (  ( {\frac {1}{312}} ( -
128j-57 ) x_0^{5}+{\frac {1}{78}} ( -29j+300
 ) x_0^{2} ) {y_0}\alpha+ ( {\frac {1}{72
}} ( -58j-11 ) x_0^{5}+\frac1{18} ( 35j+166
 ) x_0^{2} ) {y_0} ) X+ ( \frac1{13}
 ( 22j+23 ) x_0^{4}+ ( 10j-1 ) {x_0} ) \alpha+{\frac {1}{52}} ( 283j+196 ) x_0^{4}+\frac1{13} ( 124j-245 ) {x_0}\big)\big({X}^{2}+ ( 
 ( {\frac {1}{936}} ( 227j+76 ) x_0^{5}+{
\frac {1}{117}} ( -11j-265 ) x_0^{2} ) {
y_0}\alpha+ ( {\frac {1}{72}} ( 11j-8 ) {{
x_0}}^{5}+\frac19 ( -20j-28 ) x_0^{2} ) {
y_0} ) X+ ( {\frac {1}{52}} ( j-22 ) {{\it 
x0}}^{4}+1/13 ( -53j-4 ) {x_0} ) \alpha+\frac1{26}
 ( -21j-45 ) x_0^{4}+\frac1{13} ( -201j-24
 ) {x_0}\big)
$}

\noindent
est à multiplication réelle par $\Q(\sqrt{13})$.
\end{prop}
\end{expl}

Dans les trois dernières sections qui suivent, on a simplement l'existence d'une courbe et non plus une famille de courbes. Une fois les revêtements $H\rightarrow\Pp^1$ déterminés, on peut obtenir sans peine la courbe $Y$ puis le quotient $X$ de manière analogue aux types \hyperref[type:1122]{$(1,1,2,2)$}
 et \hyperref[type:1122]{$(2,2,3,3)$}. Les fonctions dont on prend les racines $l$-ième dans l'extension de Kummer sont de degré trop important, pour pouvoir être utilisées, même sur les corps finis. On propose ici d'utiliser, comme exposé ci-dessus, des points de $l$\ti{}torsion stables par certains automorphismes de la courbe $H$.

\subsection{Type (1,1,2)}\label{type:112}
On commence comme précédemment avec la formule d'Hurwitz : si $H\mapsto\Pp^1$ est un revêtement de degré 6 de type $(1,1,2)$, alors on a aussi deux possibilités ;
\begin{align*}
g(H)&=\tfrac12(2+4(2g(\Pp^1)-2)+2(6-1)+2(3-1))=2\\[-5pt]
\intertext{ou}\\[-30pt]
g(H)&=\tfrac12(2+4(2g(\Pp^1)-2)+2(6-1)+(3-1))=1.
\end{align*}
On écarte la deuxième solution pour les mêmes raisons que dans le type \hyperref[type:1122]{$(1,1,2,2)$}. 
\begin{prop}
Soit $H$ la courbe hyperelliptique d'équation $y^2=x^6-1$, de genre~$2$. On définit un revêtement de $\Pp^1$ par  $(x,y)\mapsto y$. Il est de type $(1,1,2)$, de degré $6$, comme l'automorphisme $(x,y)\mapsto(-jx,y)$ qui engendre son groupe de Galois.
\end{prop}

\begin{proof}
Le revêtement est ramifié en les $y$ tels que $y^2-1$ vaut 0 ou $\infty$, \cad pour $y=\pm1$, qui sont ramifiés d'ordre 6, et pour $y=\infty$ dont la fibre se compose des deux points à l'infini de $H$, d'ordre $3$. On a donc bien un revêtement de type $(1,1,2)$.
\end{proof}

La courbe hyperelliptique $H$ n'est pas simple et on peut adapter le résultat de la section précédente~\ref{type:1122} pour trouver un revêtement de degré $l$ non ramifié ainsi que le quotient explicite $X$.

On se propose à la place de chercher un point de la jacobienne de $H$ qui soit d'ordre $l$ tel que son sous groupe soit stable par l'action du groupe de Galois $\alpha:(x,y)\mapsto(-jx,y)$. En effet, si l'on dispose d'un tel point $P$ alors, on a une fonction $\phi(x,y)$ dont le support est $l$ fois celui du \og diviseur \fg de ce  point, au sens où on l'a défini par la formule~(\ref{div:jac}).
Le fait que le groupe engendré par ce point $P$ est stable par l'automorphisme $\alpha$ assure que l'on peut trouver $a$ d'ordre 6 dans $(\Zz l)^*$ tel que $\alpha(P)=aP$, ce qui montre l'existence d'une fonction $\Phi(x,y)$ telle que\vspace*{-5pt}
\[\phi(-jx,y)=\phi(x,y)^a\Phi(x,y)^l.\vspace*{-5pt}\]
Cela nous permet de considérer simplement la courbe \vspace*{-5pt}
\[
(Y)\ \Bigg\{\begin{array}{l@{\,=\,}l}
y^2&x^6+1\\
t^l&\phi(x,y)
.\end{array}\vspace*{-7pt}
\]
ainsi que le revêtement $Y\rightarrow\Pp^1$, $(x,y,t)\mapsto y$, qui est de type $(1,1,2)$, de groupe de Galois $G_{l,6}$ et qui engendré par les automorphismes\vspace*{-10pt}
\begin{align*}
\alpha:(x,y,t)&\mapsto\big(-jx,y,t^a\Phi(x,y)\big)\\
\sigma:(x,y,t)&\mapsto(x,y,\zeta_l t).\\[-25pt]
\end{align*}
Pour déterminer une équation plane du quotient $X$, on prend le résultant\vspace*{-8pt}
\[\Res_t\big(t^l-\phi(x,y),z-\sum_{i=0}^5\alpha^i(t)\big).\vspace*{-6pt}\]
C'est un polynôme en $y$, $z$ et $x^6$, dont il suffit de considérer le reste en $x$ par $x^6=y^2+1$, pour obtenir une équation plane de $X$ en les variables $y$ et $z$.
\vspace*{-4pt}
\begin{expl}Dans le cas $l=13$, on vérifie sans mal, à l'aide d'un logiciel tel Magma~\cite{magma}, la proposition suivante.
\end{expl}

\begin{prop}
Soit $x_0$ une racine du polynôme $x^{24}-4x^{18}+6x^{12}-\tfrac{13}{64}x^6+\tfrac{13}{256}$ et $y_0$ vérifiant $y_0^2=x_0^6-1$. Alors, le point de la jacobienne\vspace*{-8pt} 
\[P:=(x_0,y_0)+(-x_0,-y_0)-P_{\infty,1}-P_{\infty,2}\vspace*{-8pt} \] est d'ordre $13$ et vérifie de plus, en notant $\alpha(P)=(jx_0,-y_0)+(-jx_0,y_0)-P_{\infty,1}-P_{\infty,2}$, $\alpha(P)=4P$,
où l'on a choisi $j=\tfrac1{81}(-128x_0^{18} + 528x_0^{12} - 780x_0^6 - 25)$.
\end{prop}

\begin{rmq}
Pour aboutir à ce résultat, on utilise le logiciel Magma, qui, via la commande \texttt{EndomorphismRing}, donne l'anneau des endomorphismes de la jacobienne \og analytique \fg{} de la courbe $H$ d'équation $y^2=x^6-1$. Dès lors, il suffit de chercher un point du réseau dont l'action par l'endomorphisme de degré 6 soit la multiplication par $a=4$ par exemple. Cela se fait en résolvant des équations linéaires en des entiers. Ensuite, on repasse  à la jacobienne algébrique par la fonction \texttt{FromAnalyticJacobian}, puis on détermine, à l'aide de Pari~\cite{pari} et \texttt{algdep} par exemple, une extension dans laquelle pourraient être définies les cordonnées des points ainsi trouvés. Cela permet d'écrire le genre d'énoncé de la proposition ci-dessus, que l'on peut ensuite vérifier de façon algébrique, directement. Si les calculs pour déterminer $X$ semblent demander trop de puissance, on peut néanmoins utiliser cette proposition sur les corps finis. En effet, pour $p=139$, tout est défini sur $\F_p$ et on trouve finalement pour équation de $X$ la courbe hyperelliptique\vspace*{-8pt} 
\[y^2=2x^6 + 106x^5 + 21x^4 + 13x^3 + 77x^2 + 92x.\vspace*{-8pt} \]
Le polynôme caractéristique de la trace du Frobenius est $r^2+12r-16$, de discriminant réduit $4\cdot 13$,
d'où la multiplication réelle par $\Q(\sqrt{13})$.
\end{rmq}

\subsection{Type (1,1,4)}\label{type:114}
Ici, on considère $l\equiv1\mod8$. On travaille sur des corps possédant une racine carrée de $i$, que l'on note $\zeta_8$, \cad sur $\Q(\zeta_8)$ ou, par exemple, sur des corps fini de caractéristique $p$ congru à 1 modulo 8. On veut un recouvrement de degré 8, ramifié en trois points, deux d'ordre 8 et un d'ordre 2. La formule d'Hurwitz  donne $g(H)=\tfrac12\big(2+8(2g(\Pp^1-2)+2(8-1)+k(2-1)\big),$ avec $k\in\{2,4\}$ le nombre de points dans la fibre au-dessus du troisième point de ramification, nécessairement pair. La seule possibilité, pour avoir  un revêtement non ramifié d'ordre $l$  avec $g(H)\neq1$, est $k=4$ et $g(H)=2$.
\begin{prop}
Soit $H$ la courbe hyperelliptique de genre $2$, définie par l'équation $y^2=x^5+x$. Alors le revêtement $H\rightarrow\Pp^1$, $(x,y)\mapsto x^4$ est de type $(1,1,4)$, de degré $8$ tout comme l'automorphisme $(x,y)\mapsto(ix,\zeta_8y)$, qui engendre son groupe de Galois.
\end{prop}

\begin{proof}
En effet, ce revêtement est ramifié en $0, \infty$ et $1$. Pour les deux premiers, la fibre ne comporte qu'un seul point, respectivement $(0,0)$ et le point à l'infini. En $-1$, la fibre comporte les quatre points de \WS restants, identifiés par $(x,y)\mapsto x^4$.
\end{proof}

Comme précédemment, $H$ n'est pas simple et on peut encore utiliser les mêmes techniques pour construire un revêtement non ramifié $Y$, puis le quotient $X$.

On peut aussi adapter l'algorithme~\ref{algo:mult:reelle6} afin d'utiliser un revêtement $Y$ \og plus simple \fg{}, du type exposé à la section précédente. Notons d'une part que la jacobienne de $y^2=x^5+x$ peut posséder un point d'ordre $l$ sur $\F_p$, sans que ses facteurs elliptiques n'en possèdent. D'autre part, si le nombre de points sur $\F_p$ de la jacobienne est de la forme $kl$ avec $l$ ne divisant pas $k$, alors, on est sûr que le sous-groupe d'ordre $l$ est stable par $\alpha:(x,y)\mapsto(ix,\zeta_8y)$. Enfin, pour trouver un point d'ordre $l$, on peut utiliser des \og tordues \fg{} de $H$, d'équations $y^2=x^5+ax$.  

\begin{expl} Pour $l=17$, à l'aide de Magma et de façon similaire à la section précédente, on aboutit à la proposition suivante.
\end{expl}

\begin{prop}
Soit $H$ la courbe hyperelliptique $y^2=x^5+x$. Il existe un point $P$ de la jacobienne de $H$, défini en représentation de Mumford sur une extension de degré $64$, d'ordre $17$ tel que $\alpha(P)=-2P$. Il est défini en langage Magma, à l'adresse \emph{\urlb{www.normalesup.org/~iboyer/files/pt17stable.m}}.
\end{prop}

\begin{expl}
Pour $p=137$, la jacobienne de la courbe hyperelliptique $y^2=x^5+3x$ possède $2\cdot17^2\cdot41$ points sur $\F_p$. On vérifie que, en représentation de Mumford,
$P=(x^2 + 41x + 11, 98x + 84)$, 
est d'ordre 17, et satisfait $\alpha(P)=-2P$. On trouve finalement pour $X$ la courbe hyperelliptique de genre 2
\[%
y^2=x^6 + 10x^5 + 3x^4 + 103x^3 + 20x^2 + 10x + 120\vspace*{-2pt}
\] et le polynôme caractéristique de la trace du Frobenius,  $r^2-9\cdot17$, définit, malgré~des apparences \og supersingulières \fg{},  le corps de nombres $\Q(\sqrt{17})$.

Comme on l'a déjà remarqué, le sous-groupe des points d'ordre 41 est nécessairement stable par $\alpha$, et on vérifie par exemple que pour $P=(x^2 + 70x + 60, 84x + 18)$, 
on a $\alpha(P)=3P$. Cela nous permet de calculer une équation de $X$ de genre 5, à multiplication réelle par $\Q(\zeta_{41}^{(8)})$. Le calcul est très rapide, alors qu'il serait certainement impossible en utilisant les techniques exposées dans les sections~\ref{type:1122} et \ref{type:2233}. On obtient une courbe de genre 5 que l'on peut exprimer comme l'intersection de 3 quadriques sur $\Pp^4$. Il suffit, comme on l'a déjà fait plusieurs fois avec les courbes de genre 3, de résoudre des équations linéaires en les coefficients d'une quadrique homogène générale en 5 variables, que l'on évalue sur une base de différentielles holomorphes. Parmi les quadriques que l'on trouve, on en choisit 3 \og indépendantes \fg, ce que l'on peut vérifier en calculant, par exemple, la dimension de leur intersection dans $\Pp^4$. On trouve les équations quadriques suivantes en $a,b,c$ et $d$, de degré total 2.

{\footnotesize
\begin{align*}
q_1(a,b,c,d)&=ac+105ad+118b^2+107bc+136bd+126c^2+113cd+98d^2+132a+40b+47c+69d+1\\
q_2(a,b,c,d)&=ab+77ad+101b^2+17bc+45bd+131c^2+36cd+69d^2+121a+113c+33d+14\\
q_3(a,b,c,d)&=a^2+122ad+107b^2+132bc+65bd+90c^2+111cd+47d^2+91a+131b+73c+128d+59
\end{align*}
}%
\end{expl}

\subsection{Type (1,2,5)}\label{type:125}
Pour ce dernier cas, on a nécessairement $l\equiv1\mod10$. On travaille sur des corps possédant une racine cinquième de l'unité, que l'on note $\zeta_5$, \cad sur $\Q(\zeta_5)$ ou, par exemple, sur des corps finis de caractéristique $p$ congrue à 1 modulo 5.

On veut ici un recouvrement de degré 10, ramifié en trois points, d'ordre 10, 5 et 2. La formule d'Hurwitz nous donne
\[g(H)=\frac12\big(2+10(2g(\Pp^1-2)+(10-1)+r(5-1)+s(2-1)\big)\]
avec $r\in\{1,2\}$ le nombre de points dans la fibre au-dessus du point de ramification d'ordre 5 et $s\in\{1,3,5\}$ pour la fibre d'ordre 2, nécessairement impair. La seule possibilité, pour avoir ensuite un revêtement non ramifié d'ordre $l$ qui ne soit pas une courbe elliptique, est $r=2$, $s=5$ pour $g(H)=2$.

\begin{prop}
Soit $H$ la courbe hyperelliptique de genre $2$, définie par l'équation $y^2=x^5+1$. Alors le revêtement $H\rightarrow\Pp^1$, $(x,y)\mapsto y^2$ est de type $(1,2,5)$, de degré $10$ tout comme l'automorphisme $(x,y)\mapsto(\zeta_5 x,-y)$, qui engendre son groupe de Galois.
\end{prop}

\begin{proof}
En effet, ce revêtement est ramifié en $0$, $\infty$ et $1$. En 0, la fibre comporte les 5 points de \WS $((-\zeta_5)^k,0)$, en 1, la fibre est composée des deux points $(-1,0)$ et $(1,0)$ tandis qu'en l'infini, la fibre ne comporte que le dernier point de \WS, le point à l'infini. 
\end{proof}

À partir d'un point d'ordre $l$ sur la jacobienne de $H$, on trouve une fonction $\psi(x,y)$ sur $H$ dont  chacun des points du diviseur est d'ordre multiple de $l$. Comme précédemment, on cherche des points qui sont stables par $\alpha:(x,y)\mapsto(\zeta_5x,-y)$, le signe moins sur l'ordonnée s'obtenant par $[-1]$ sur la jacobienne.

La jacobienne de la courbe $H$ est simple, à multiplication complexe par $\Q(\zeta_5)$, ce qui facilite beaucoup la recherche d'un tel point. La jacobienne de $H$ peut être vue comme $\C^2/\Phi(\Oc)$ où $\Oc$ est l'anneau des entiers de $\Q(\zeta_5)$ et $\Phi$ est donné par le \tcm. On vérifie que le réseau 
\[\begin{pmatrix}
\zeta_5&\zeta_5^3&1&\zeta_5^2 + 1\\\zeta_5^2&\zeta_5^6&1&\zeta_5^4 + 1
\end{pmatrix}\]
convient. Soit maintenant $l$ un nombre premier vérifiant de plus $l\equiv1\mod 10$. Alors, l'idéal $(l)$ est décomposé dans $\Q(\zeta_5)$. Soit $a$ d'ordre 5 dans $(\Zz l)^*$, si bien que l'idéal $(\zeta_5-a)$ possède un facteur premier, et principal, $(\pi)$ qui divise $(l)$. Ainsi,
\[\frac{\zeta_5}{\pi}=\frac{a}{\pi}\mod\Oc\]
si bien que le point de la jacobienne $\tfrac1\pi$ est d'ordre $l$, et le sous-groupe qu'il engendre est stable par la multiplication complexe par $\zeta_5$. 
Grâce à ce point de la jacobienne, on construit une fonction $\phi(x,y)$ dont le diviseur est $l$ fois le diviseur de ce point, au sens~(\ref{div:jac}). On peut alors prendre simplement comme extension $t^l=\phi(x,y)$  puisque par construction, on a l'existence d'une fonction $f(x,y)$ vérifiant $\phi(\zeta_5x,y)=\phi(x,y)^af(x,y)^l$
et on en déduit, de façon classique sur les ordonnées, une fonction $\Phi(x,y)$ telle que
\[\phi(\zeta_5x,-y)=\phi(x,y)^a\Phi(x,y)^l.\]
Le revêtement de $H$ par la courbe $Y$, définie par l'extension de Kummer $t^l=\phi(x,y)$, donne un revêtement total sur $\Pp^1$ galoisien dont le groupe de Galois $G_{l,10}$ est engendré par les automorphismes
\begin{align*}
\alpha:(x,y,t)&\mapsto\big(\zeta_5x,-y,t^a\Phi(x,y)\big)\\
\sigma:(x,y,t)&\mapsto(x,y,\zeta_l t).
\end{align*}
On obtient ensuite, comme depuis le début,  la courbe $X$ en prenant le résultant
\[\Res_t\big(t^l-\phi(x,y),z-\sum\limits_{i=0}^5\alpha^i(t)\big).\] Ce résultant
s'exprime en $y^2$ et $x^5$, ce qui fournit l'équation de $X$ en prenant le reste par $y^2-(x^5+1)$ d'une part, puis par $x^5-w$ d'autre part.

\begin{expl}Dans un premier temps, on cherche des points de la jacobienne de $y^2=x^5+1$ d'ordre $l$ et dont le sous-groupe engendré est stable par la multiplication complexe. On donne la proposition suivante pour $l=11$ et $l=31$, qui se vérifie facilement avec Magma par exemple.
\end{expl}

\begin{prop} Les points suivants engendrent des groupes d'ordre $l$, stables par $(x,y)\!\mapsto\!(\zeta_5x,y)$.
\begin{enumerate}
\item On considère $l=11$ et soit $\gamma$ une racine de $t^5 + \tfrac1{121}(80\zeta_5^3 + 320\zeta_5^2 + 1040\zeta_5 + 1264)$ et $\delta$ une racine de $t^2 + \tfrac{1}{11}(-36\zeta_5^3 - 12\zeta_5^2 - 28\zeta_5 + 1)$. Alors, le point 
\[P=\big(t^2 + \tfrac14(\zeta_5^3 + 2\zeta_5)\gamma^3t + \gamma,
\tfrac{1}{22}(-\zeta_5^3 + \zeta_5^2 + 8\zeta_5 - 6)\gamma^2\delta t + \delta\big)\]
est d'ordre $11$ et vérifie $\alpha(P)=2P$.
\item On considère $l=31$. Il existe un point $P$ de la jacobienne de $y^2=x^5+1$, donné en représentation de Mumford sur une extension de degré $120$, tel que $31P=0$ et $\alpha(P)=-2P$. Sa définition est donnée à l'adresse \emph{\urlb{www.normalesup.org/~iboyer/files/pt31stable.m}}, dans le langage Magma.
\end{enumerate}
\end{prop}

Toutefois, le degré des extensions considérées semble trop important pour calculer un modèle plan de la courbe $X$. Néanmoins, on peut utiliser ces résultats pour réduire les calculs sur $\F_p$. Notons que l'on peut aussi utiliser les jacobiennes des courbes $y^2=x^5+a$ dont le nombre de points sur $\F_p$ est divisible par $l$ mais pas par $l^2$.

\begin{expl} Pour $l=31$, on se place pour $\F_{61}$. La jacobienne de $y^2=x^5+8$ possède $11^2\cdot31$ points, ce qui assure la stabilité du sous-groupe d'ordre 31. On trouve pour $X$ l'équation quartique
{\footnotesize
\[{u}^{4}+30{u}^{3}v+2{u}^{2}{v}^{2}+26u{v}^{3}+41{v}^{4}+46{u
}^{3}+50{u}^{2}v+37u{v}^{2}+17{v}^{3}+14{u}^{2}+21uv+16{v}
^{2}+35u+8v+32
\]
}
dont le polynôme caractéristique de la trace du Frobenius est $r^3+34r^2+375r+1334$, définit le corps de nombres $\Q(\zeta_{31}^{(10)})$.
\end{expl}
 
\section{Résumé des résultats}\label{chap2:sec:fin}
Pour conclure résumons dans la table~\ref{res:table} les principaux résultats obtenus ci-dessus, concernant des courbes explicites à multiplication réelle dont J. Ellenberg montre l'existence dans~\cite{ellen}.


\begin{table}[H]
\centering
\begin{tabular}{|c||c|c|c|}
\hline
\multirow{2}{*}{$n,l$}&\multirow{2}{*}{Courbes exposées}&Corps de&\multirow{2}{*}{Existence théorique}\\
&&définition&
\\\hline&&&\\[-1.3em]\hline
\multirow{3}{*}{$\Q(\zeta_l^+)$}&famille à 1 paramètre~\cite{top}&$\Q$&\multirow{3}{*}{famille à 3 paramètres}\\
&famille à 2 paramètres&$\Q$&\\
&famille à 3 paramètres&extension&\\\hline
\multirow{2}{*}{$\Q(\zeta_l^{(4)})$}&1 courbe&$\Q(i)$&\multirow{2}{*}{famille à 1 paramètre}\\[-3pt]
&famille à 1 paramètre&extension&\\
\hline
\multirow{2}{*}{$\Q(\zeta_l^{(6)})$}&famille à 1 paramètre&extension&\multirow{2}{*}{famille à 1 paramètre}\\[-3pt]
&1 courbe&extension&\\
\hline
$\Q(\zeta_l^{(8)})$&1 courbe&extension&1 courbe\\
\hline
$\Q(\zeta_l^{(10)})$&1 courbe&extension&1 courbe\\
\hline
\end{tabular}
\caption{Courbes explicites à multiplication réelles par des sous-corps de cyclotomiques}
\label{res:table} 
\end{table}

%% file: crb_mult_reelle.bbl
\begin{thebibliography}{{Mes}~91+}

\bibitem[Ell~01]{ellen}
{\bfseries J.~S. Ellenberg}.
\newblock Endomorphism Algebras of Jacobians.
\newblock {\itshape Advances in Mathematics}, Vol.~162, \No{}2, pp.~243 -- 271,
  2001.
\newblock \urlb{www.math.wisc.edu/~ellenber/EndJacAjour.pdf}.

\bibitem[Fri~24]{fricke}
{\bfseries R.~Fricke}.
\newblock {\itshape Lehrbuch der Algebra}.
\newblock Vieweg u. Sohn, 1924.

\bibitem[GP~13]{pari}
\textbf{User's Guide to PARI / GP}.
\newblock 2013.
\newblock \urlb{pari.math.u-bordeaux.fr}.

\bibitem[Lep~91]{leprevost91}
{\bfseries F.~Lepr\'{e}vost}.
\newblock Familles de courbes de genre 2 munies d'une classe de diviseurs
  rationnels d'ordre 15, 17, 19 ou 21.
\newblock {\itshape C. R. Acad. Sci. Paris}, Vol.~313, \No{}I, pp.~771--774,
  1991.

\bibitem[Lep~95]{leprevost95}
{\bfseries F.~Lepr\'{e}vost}.
\newblock Jacobiennes de certaines courbes de genre 2 : torsion et
  simplicit\'{e}.
\newblock {\itshape Journal de Th\'{e}orie des Nombres de Bordeaux}, Vol.~7,
  pp.~283--306, 1995.
\newblock \urlb{www.numdam.org/item?id=JTNB_1995__7_1_283_0}.

\bibitem[Mag~13]{magma}
\textbf{Handbook of Magma functions}.
\newblock 2013.
\newblock Edition 2.19(2013), \urlb{magma.maths.usyd.edu.au/magma/}.

\bibitem[Map~13]{maple}
\textbf{Maple}\protect\re.
\newblock 2013.
\newblock \urlb{www.maplesoft.com/products/Maple/}.

\bibitem[{Mes}~91]{mestrehyp}
{\bfseries J.-F. {Mestre}}.
\newblock Familles de courbes hyperelliptiques à multiplications réelles.
\newblock In: {\itshape Arithmetic Algebraic Geometry, {\emph{Progr. Math.}}},
  Birkhäuser, Boston, 1991.

\bibitem[Ogg~74]{ogg}
{\bfseries A.~P. Ogg}.
\newblock Hyperelliptic modular curves.
\newblock {\itshape Bulletin de la Société Mathématique de France},
  Vol.~102, pp.~449--462, 1974.
\newblock \urlb{www.numdam.org/item?id=BSMF_1974__102__449_0}.

\bibitem[TTV~91]{top}
{\bfseries W.~Tautz, J.~Top et A.~Verberkmoes}.
\newblock Explicit Hyperelliptic Curves with Real Multiplication and
  Permutation Polynomials.
\newblock {\itshape Canad. J. Math}, Vol.~43, \No{}5, pp.~1055--1064, 1991.

\end{thebibliography}
